%% file: berger_klosin_2016__1_.tex
\documentclass{amsart}

\usepackage{amsfonts}
\usepackage{amscd}
\usepackage{amsthm}
\usepackage{amssymb}
\usepackage{amsmath}
\usepackage{epsfig}
\usepackage{graphicx}
\usepackage{hyperref}

\include{makros}

\include{settings}

\begin{document}
\author{Tobias Berger and Krzysztof Klosin}
\title{A $p$-adic Hermitian Maass lift}
\subjclass[2010]{11F33, 11F50, 11F55}

\keywords{$p$-adic automorphic forms, Jacobi forms, Hermitian modular forms}

\thanks{The work of the first author was supported by the EPSRC First Grant EP/K01174X/1. The work of the second author was supported by a grant from the Simons Foundation (\#354890, Krzysztof Klosin). In addition the second author was partially supported by a PSC-CUNY Award, jointly funded by The Professional
Staff Congress and The City University of New York.}

\begin{abstract}

For $K$ an imaginary quadratic field with discriminant $-D_K$  and associated quadratic Galois character $\chi_K$, Kojima, Gritsenko and Krieg  studied a Hermitian Maass lift of elliptic modular cusp forms of level $D_K$ and nebentypus $\chi_K$ via Hermitian Jacobi forms to Hermitian modular forms of level one for the unitary group $U(2,2)$ split over $K$. 
We generalize this 
 (under certain conditions on $K$ and $p$)
 to the case of $p$-oldforms of level $pD_K$ and  character $\chi_K$. To do this, we define an appropriate Hermitian Maass space for general level and prove that it is isomorphic to the space of special Hermitian Jacobi forms. 
We then show how to adapt this construction to lift a Hida family of modular forms to a $p$-adic analytic family of automorphic forms in the Maass space of level $p$.
\end{abstract}

\maketitle

\section{Introduction}
Since the groundbreaking work of Hida \cite{HIda86}, there has been a lot of interest in $p$-adic families of modular forms. While interesting in their own right, their use was also instrumental in proving the Iwasawa Main conjecture for $\bfQ$ and totally real fields \cite{MazurWiles84, Wiles90}. More recently, analogous $p$-adic families have been studied for automorphic forms on higher-rank reductive algebraic groups, cf. e.g.,  \cite{Taylorthesis, Chenevier04, Urban11, AndreattaIovitaPilloni15}. 
Such families were used by Skinner and Urban to prove the Iwasawa Main Conjecture for $\GL(2)$ \cite{SkinnerUrban14}. 

In \cite{Kawamura10preprint} Kawamura provides a construction of a $p$-adic family of Ikeda lifts 
from $\GL(2)$ to $\GSp(2n)$ for modular forms of level one. One of the crucial elements of his construction is the existence of a $\Lambda$-adic  Shintani 
 lifting (i.e., a $p$-adic family of such lifts) proved by Stevens \cite{Stevens94}, which associates a $p$-adic family of modular forms of half-integer weight to a $p$-adic family of  modular forms on $\GL(2)$ and relies on interpolating the cycle integrals which express Fourier coefficients. When $n=2$, the Ikeda lift is the same as the Saito-Kurokawa lift. 
 Kawamura's result in  \cite{Kawamura10preprint} generalized previous results on $p$-adic interpolations, 
 in particular, by Guerzhoy \cite{Guerzhoy00}, who proved a $p$-adic interpolation of an essential part of the Fourier expansion of the Saito-Kurokawa lift by using its construction as a combination of the Shintani lifting with the Maass lifting of Jacobi forms to Siegel modular forms.

In this paper we study the 
 Hermitian  Maass lift, which associates an automorphic form $F_f$ on the quasi-split unitary group $\U(2,2)$ to a modular form $f$ on $\GL(2)$, and 
 construct a suitable $p$-adic family of such lifts. Our construction is 
different 
from that of  \cite{Guerzhoy00} and \cite{Kawamura10preprint} as we now explain. 
On the one hand, as opposed to the Ikeda lift, the Fourier coefficients of $F_f$ can be expressed 
explicitly by the Fourier coefficients of $f$. This allows us to use the known interpolation properties of $f$ to interpolate $F_f$ in a more direct fashion than is the case in \cite{Kawamura10preprint}. 
However, two major 
problems arise. As is well-known, 
not all Fourier coefficients  of a modular form of level prime to $p$ depend $p$-adically analytically on the weight and the form needs to be ``$p$-stabilized'' first to mitigate this problem. However, this procedure produces a form $f_p$ 
of level divisible by $p$. Let 
 $-D$ be the discriminant of the imaginary quadratic field $K$ over which $\U(2,2)$ splits. We assume that $K$ has class number one. Then the only known constructions of the Maass lift \cite{Kojima, Krieg91, Gritsenko90, Ikeda08} are for $f$ a cusp form of level $D$ and character $\chi_K$, the quadratic character associated to the extension $K/\bfQ$, 
 thus not allowing us to lift $f_p$. 
 (\cite{Guerzhoy00} and \cite{Kawamura10preprint} get around this issue by $p$-stabilising the full level lift of the underlying oldform, see Remark \ref{rem4.6} for a comparison to our approach.) 
 This is one of the reasons why we devote a major part of the paper to generalizing the Maass lifting procedure to forms of level $Dp$. In fact, we restrict ourselves here to the subspace of  $p$-oldforms, which is both sufficient for our purposes ($f_p$ is old at $p$) and allows us to reduce some of the proofs to the case of level $D$. 

The second major problem arises from the fact that the family $\tilde{\mF}$ in which $f$ (or more precisely $f_p$)  lives cannot be directly lifted to a family on $\U(2,2)$ by applying the Maass lift to all the specializations of $\tilde{\mF}$. The obstacle lies in the fact that the Maass lifting
procedure ``lifts'' not $f_p$ but $f_p-f_p^c$, where $f_p^c$ is obtained from $f_p$ by applying the complex conjugation to its Fourier coefficients, and complex conjugation is not a $p$-adically continuous operation. We circumvent this problem by essentially reversing the order of these operations. More precisely we construct a different $p$-adic family $\mF$ of modular forms on $\GL(2)$ whose specializations are $p$-stabilizations of the forms $f-f^c$ which lie in the  (analogue for the Hermitian Maass lift of the Kohnen) plus-space of cusp forms of level $Dp$ and character $\chi_K$. 
We show that these specializations can be lifted to a $p$-adic analytic family of Hermitian Maass forms of level $p$, thus providing us with a version of a $\Lambda$-adic Hermitian Maass lift. We achieve this by using local properties of the Galois representation attached to the family $\tilde{\mF}$ proved in \cite{EmertonPollackWeston06}.  
Let us now explain the organization of the paper in more detail. 

Let $K$ be as above and write $\OK$ for the ring of integers of $K$. Let $k$ be a positive integer divisible by $\#\OK^{\times}$. 
We begin the paper by proposing in section \ref{The Maass space and the Jacobi forms} a definition of the Maass space for automorphic forms on $\U(2,2)$ of weight $k$ and 
level $\Gamma_0^{(2)}(N)$ for arbitrary integer $N$. 
We generalize the isomorphism between so-called ``special" Jacobi forms and 
the Maass space 
proved by \cite{Haverkamp95} Satz 7.6 from 
$N=1$ to arbitrary $N$ (see Theorem \ref{Jacobiisom}).  
This completes work on this in \cite{TY12} Proposition 2.2. We use similar arguments as \cite{Ibukiyama12} for the Maass lift to Siegel modular forms, 
but in addition we prove surjectivity of the Hermitian Maass lift onto  
 the Maass space of Hermitian modular forms we defined. Section \ref{Some transformation properties of the theta function} studies some transformation properties of theta functions and proves that there is an injective ``descent" from the special Jacobi forms of level $N$ to elliptic 
 weight $k-1$ forms of level $D N$ 
and character $\chi_K$.  Section \ref{Maass lift of p-old plusforms} is devoted to constructing a lifting from the plus-space 
of $p$-old forms (for $p$ which splits in $K$) of weight $k-1$ level $Dp$ and character $\chi_K$ 
to the space of special Jacobi forms, which combined with the results of section 2 give us the full  Hermitian  Maass lift. In section \ref{The Hecke invariance} we prove the Hecke equivariance of the Maass space and discuss the descent of the Hecke operators for $p$-ordinary eigenforms. Finally in section \ref{A prime to p interpolation} we study $p$-adic families and construct a $\Lambda$-adic Maass lift.

Note that $\Lambda$-adic liftings can be used via pullback formulae to construct $p$-adic $L$-functions (for Eisenstein series see e.g. \cite{BochererSchmidt00, HarrisLiSkinner05, EischenWanpreprint}, for Saito-Kurokawa lifting sketched in \cite{Lithesis}). 
We plan to study the application of our $\Lambda$-adic Maass lift in combination with the pullback formula of \cite{Atobe15} to the construction of $p$-adic Rankin-Selberg $L$-functions  in future work.

We would like to thank Olav Richter for providing us with a copy of the thesis of Klaus Haverkamp.

\section{The Maass space and the Jacobi forms}\label{The Maass space and the Jacobi forms}
Let $K$ be an imaginary quadratic field of discriminant $-D_K$ 
 (i.e., $K=\bfQ(i\sqrt{D_K})$)
and class number one. Write $\Oo$ for the ring of integers of $K$, 
$\chi_K$ for the associated quadratic character and 
 $\mD^{-1}=\frac{\Oo}{\sqrt{-D_K}}$ 
 for the inverse different of $K$. Set $\U(n,n)$ to be the $\bfZ$-group scheme defined by $$\U(n,n):= \left\{M\in \Res_{\Oo/\bfZ} \GL_{2n/\Oo} \mid {}^t\ov{M} \bmat &-I_n\\ I_n \emat M = \bmat &-I_n\\ I_n \emat \right\},$$ where $a \mapsto \ov{a}$ is the automorphism induced by the non-trivial element of $\Gal(K/\bfQ)$. Here $\Res_{\Oo/\bfZ}$ denotes the Weil restriction of scalars and we will write $I_n$ (resp. $0_n$) for the $n\times n$ identity matrix (resp. the $n\times n$ zero matrix). Write $M_n$ for the (additive) $\bfZ$-group scheme of $n\times n$ matrices.
For a positive integer $N$ put 
$$\Gamma_0^{(n)}(N):=\left\{\bmat A&B\\C&D \emat\in \U(n,n)(\bfZ) \mid C \in N M_n(\Oo) \right\}.$$ 
We reserve the notation $\Gamma_0(N)$ for the standard congruence subgroup of $\SL_2(\bfZ)$ whose elements have the lower-left entry divisible by $N$. This group is closely related $\Gamma_0^{(1)}(N)$ as shown by Lemma \ref{Hilfsatz2}.

 Let $\mS_n$ to be the $\bfZ$-group scheme defined by  $\mS_n=\{g \in \Res_{\Oo/\bfZ} M_{n/\Oo} \mid g={}^t\ov{g}\}.$ 
 Note that $\mS_n(\bfZ)$ is the  group of $n \times n$ hermitian matrices with entries in $\Oo$. We set $\mS_n(\bfZ)^{\vee}$ to be the lattice in $\mS_n(\bfQ)$  dual to the lattice $\mS_n(\bfZ)$, i.e., $\mS_n^{\vee}(\bfZ) = \{ g \in \mS_n(\bfQ) \mid \tr (g \mS_n(\bfZ)) \subset \bfZ\}$. 
Here $\tr: K \to \bfQ$ denotes the trace, i.e., $\tr(a) = a+\ov{a}$ and for future use we also introduce the norm $N: K \to \bfQ$ given by $N(a)=a\ov{a}$.
Since we will most frequently have a need to use $\mS_2^{\vee}(\bfZ)$ we will simply denote it by $\mS$.
One has $$\mathcal{S}:=\mathcal{S}_2^{\vee}(\bfZ) = \left\{ \bmat
\ell
& t \\ \ov{t} & m \emat \in M_2(K) \mid \ell, m \in \bfZ, t \in\mD^{-1} \right\}.$$

For $T \in \mS$ we define $\epsilon(T):={\rm max}\{q \in \bfZ_+| \frac{1}{q} T \in \mS\}$.
A holomorphic function $F$ on the Hermitian upper half space $$\bfH_2:=\{Z \in M_2(\bfC) | (Z- {}^t\ov{Z})/(2i)>0 \}$$
 satisfying $$F((AZ+B)(CZ+D)^{-1})=\det(CZ+D)^k F(Z) \, \text{ for all } \gamma=\bmat A&B\\C&D\emat \in \Gamma_0^{(2)}(N)$$ is called a Hermitian modular form of weight $k$ and  level $N$.  We denote by $\mM_k(N)$ the $\bfC$-space of all such forms. Set $e[z]:= e^{2 \pi i z}$. Any $F \in \mM_k(N)$ has a Fourier expansion $$F(Z)=\sum_{T \in \mS, T \geq 0} C_F(T) e[\tr TZ].$$

Let $k$ be a positive integer such that $\# \Oo^{\times} \mid k$.
We say that $F \in \mM_k(N)$ is in the Maass space $\mM^*_k(N)$ if there exists a function $\alpha^*_F: \bfZ_{\geq 0} \to \bfC$ satisfying $$C_F(T)= \sum_{\substack{d \in \bfZ_{+}, d \mid \epsilon(T)\\ \textup{gcd}(d,N)=1}} d^{k-1} \alpha^*_F(D_K \det T/d^2)$$ for all $T \in \mS$, $T \geq 0$, $T \neq 0_2$.

 \begin{rem} \label{2.1} Note that in fact it is enough to define $\alpha_F^*$ on those positive integers $\ell$ which satisfy  $\ell \equiv -D_K N(u)$ (mod $D_K$), where $u$ runs over $\mD^{-1}/\Oo$. 
 Indeed, we claim that every  $\ell$ of the form $D_K \det T/d^2=D_K \det T' (\epsilon(T)/d)^2 $ 
 (where  $\epsilon(T')=1$)  satisfies a congruence of the above type.  This is so because for $T' \in\mS$, we have $-D_K \det T' \subset D_K N(\mD^{-1}/\Oo) \subset \bfZ/D_K\bfZ$ and clearly all squares in $\bfZ/D_K\bfZ$ are norms of elements of $\mD^{-1}/\Oo$.   \end{rem}

We will now  recall from \cite{TY12} Section 2.2 the definition of Hermitian Jacobi forms with level. 
For this it will be useful to record the following lemma.
\begin{lemma}\label{Hilfsatz2} For any matrix $M \in \U(1,1)(\bfZ)$ there exists a matrix $A \in \U(1,1)(\bfZ) \cap M_2(\bfZ)=\SL_2(\bfZ)$ and $\epsilon \in \Oo^{\times}$  such that $M=\epsilon A$. \end{lemma}
\begin{proof} This follows directly from Hilfsatz 2 of \cite{Grosche78}.  
\end{proof}

 Let $\bfH$ denote the complex upper half-plane.
For integers $k>0$   such that $\# \Oo^{\times} \mid k$  and $m \geq 0$, 
there is an action of the Jacobi group $U(1,1)(\bfZ) \ltimes \Oo^2$ (where we write an element of $\U(1,1)(\bfZ)$ as $\epsilon A$ as in Lemma \ref{Hilfsatz2}) on functions on $\bfH \times \bfC^2$ given by
\begin{equation} \label{prop1} \begin{split}  \varphi|_{k,m}[\epsilon A]:=& (c\tau+d)^{-k}e\left[-  m\frac{czw}{c\tau+d}\right]\varphi\left(\frac{a\tau+b}{c\tau+d}, \frac{\epsilon z}{c\tau+d}, \frac{\ov{\epsilon} w}{c\tau+d}\right)\\
\varphi|_{m}[\lambda, \mu] := & e[ m (N(\lambda)\tau + \ov{\lambda}z + \lambda w)]\varphi(\tau, z+\lambda \tau + \mu, w +\ov{\lambda} \tau + \ov{\mu}).\end{split}\end{equation} 

\begin{rem} Note that the decomposition $M=\epsilon A$ as in Lemma \ref{Hilfsatz2} is not unique, but if $\epsilon A = \epsilon' A'$ are two different decompositions, then we must have $\epsilon'=-\epsilon$ and $A=-A'$. Since $\# \Oo^{\times} \mid k$, i.e.,  in particular $k$ is even, the action in (\ref{prop1}) is well-defined. \end{rem}

For integers $k>0$ and $m \geq 0$,
let $J_{k, m}(N)$ denote the space of Jacobi forms of weight $k$, index $m$ and level $N$. Such forms $\varphi$ are holomorphic functions on $\bfH \times \bfC^2$ required to satisfy the following conditions:
\begin{itemize} \item $\varphi|_{k,m}[\epsilon A]=\varphi$ for all $A \in \Gamma_0(N)\subset {\rm SL}_2(\bfZ)$ 
and $\epsilon \in \Oo^{\times}$ and 
$\varphi|_{m}[\lambda, \mu]=\varphi$ for all $\lambda, \mu \in \Oo$.
\item For each $M \in {\rm SL}_2(\bfZ)$, $\varphi|_{k,m} [M]$ 
 has a Fourier expansion (see e.g. \cite{TY12} p. 1953) of the form 
$$(\varphi|_{k,m} [M])(\tau, z, w)=\sum_{\substack{\ell \in \mathbf{Z}_{\geq 0}, t \in \mD^{-1}\\ \nu N(t) \leq \ell m}} c_{\varphi}^M(\ell, t) e[\frac{\ell}{\nu} \tau + \ov{t} z + t w],$$ where $\nu \in \bfZ_+$ depends on $M$ (and equals $1$ for $M=I_2$). 
 For $M=I_2$ 
 we write $c_{\varphi}(\ell, t):=c_{\varphi}^M(\ell, t)$. 
\end{itemize}

For a positive integer $m$, define $$\Delta_N(m):= \left\{\bmat a&b \\ Nc & d\emat \mid a,b,c,d \in \bfZ, ad-bcN=m, \textup{gcd}(a,N)=1\right\}.$$ 
Following \cite{Kojima} we extend the action of $\SL_2(\bfZ) \subset \U(1,1)(\bfZ)$ on functions on $\bfH \times \bfC^2$ defined in (\ref{prop1}) to that of $\GL_2(\bfR)^+$ (the plus indicates positive determinant) by setting  
$$\varphi|_{k,t}[S] (\tau, z, w):= \varphi\left(\frac{a\tau+b}{c\tau+d}, \frac{\sqrt{\det S} z}{c\tau +d},  \frac{\sqrt{\det S} w}{c\tau +d}\right)e\left[-\frac{ tczw}{c\tau+d}\right](c\tau+d)^{-k}$$ for any $S=\bmat a&b \\ c& d \emat \in \GL_2(\bfR)^+$.
For $t=1$ we also write $$\varphi|[S]_{k} (\tau, z, w):=\varphi|_{k,1}[S] (\tau, z, w).$$
Define the index shifting operator $$V_m: J_{k,t}(N) \to J_{k, mt}(N)$$ by $$(V_m\varphi)(\tau, z,w) := m^{k-1} \sum_{g \in \Gamma_0(N) \setminus \Delta_N(m)} (\varphi|_{k,t}[g])(\tau, z,w).$$ 
 Since $\varphi$ is  invariant under $\Gamma_0(N)$ we see that $(V_m\varphi)(\tau, z,w)$ is well-defined.

 For $\varphi \in J_{k,1}(N)$ we also define  $$(V_0 \varphi)(\tau, z, w):= c_{\varphi}(0,0) \left ( -\frac{B_k}{2k}\prod_{p \mid N} (1-p^{k-1})+ \sum_{n \in \bfZ_+} \sum_{\substack{d \mid n\\ \textup{gcd}(d,N)=1}} d^{k-1} e[n \tau]\right ),$$ where $B_k$ denotes the $k$th Bernoulli number.

The Fourier expansion of $F \in \mM_k(N)$ can be rewritten as a Fourier-Jacobi expansion as follows: $$F(Z) = \sum_{m=0}^{\infty}\varphi_m(\tau, z, w) e[ m \tau^*],$$ where $$Z=\bmat \tau& z \\ w & \tau^*\emat \in \bfH_2.$$ 
As in \cite{Haverkamp95} Satz 7.1 we have that the $m$-th Fourier-Jacobi coefficients $\varphi_m$ lies in $J_{k,m}(N)$.

\begin{prop} \label{unique} If $F\in \mM_k^*(N)$, then it is uniquely determined by $\varphi_1$. \end{prop}

\begin{rem} For the proof of Proposition \ref{unique} we adapt the arguments of section 3 of \cite{Kojima} carried out for level one and $K=\bfQ(i)$ to our more general situation. We decided to include a detailed proof as the account in \cite{Kojima} is very brief. We also took the opportunity to correct a few small errors in \cite{Kojima}.
\end{rem}
\begin{proof}
We will show that the map $F \mapsto \varphi_1$ defines an injection $\mM_k^*(N) \hookrightarrow J_{k,1}(N)$. 
Recall the Fourier and Fourier-Jacobi expansions
\begin{eqnarray*} F\left(\bmat \tau & z \\ w & \tau^*\emat \right) &=& \sum_{T\in\mS, T \geq 0}C_F(T) e\left[\tr T\bmat \tau & z \\ w & \tau^*\emat \right]\\&=& \sum_{m=0}^{\infty}\varphi_m(\tau, z, w) e[ m \tau^*].\end{eqnarray*} Note that if we write 
$T=\bmat n & \frac{\alpha}{\sqrt{-D_K}}\\ -\frac{\ov{\alpha}}{\sqrt{-D_K}}& m\emat \in \mS$ with $n,m \in \bfZ$, $\alpha \in \Oo$, we have $\det T = nm-\frac{|\alpha|^2}{D_K}$,  hence $T\geq 0$ 
 if and only if all of the following are satisfied: $n,m\in \bfZ_+$ and $\alpha \in \Oo$ is such that $|\alpha|^2\leq D_Knm$. 
Because of this inequality the $0$-th  Fourier-Jacobi coefficient of $F \in \mM^*_k(N)$ picks out the terms for $T=\bmat n&0\\0&0 \emat$ with $n \geq 0$ and, by using the Maass condition for $n \geq 1$, is  given by
 $$\varphi_0(\tau,z,w)=C_F(0_2)+\alpha_F^*(0) \sum_{n=1}^{\infty} \sum_{\substack{d \in \bfZ_+, d \mid n\\ \textup{gcd}(d,N)=1}} d^{k-1} e[n \tau].$$ 
Furthermore this implies \begin{eqnarray*}F\left(\bmat \tau & z \\ w & \tau^*\emat \right) &=& C_F(0_2)+\alpha_F^*(0) \sum_{n=1}^{\infty} \sum_{\substack{d\in \bfZ_+, d \mid n\\\textup{gcd}(d,N)=1}} d^{k-1} e[n \tau]  \\ &+&\sum_{n=1}^{\infty} \sum_{m=1}^{\infty} \sum_{\substack{\alpha \in \Oo\\ |\alpha|^2 \leq D_Knm}} \sum_{\substack{d \in \bfZ_+\\ d\mid n, d\mid m, d\mid \alpha\\\textup{gcd}(d,N)=1}} d^{k-1} \alpha^*_F\left(\frac{D_Knm-|\alpha|^2}{d^2}\right)e\left[\tr T\bmat \tau & z \\ w & \tau^*\emat \right].\end{eqnarray*}
Here and below by writing $d \mid \alpha$ we mean that $d^{-1} \alpha \in \Oo$. 
As the  $0$-th Fourier-Jacobi coefficient restricted to $\bfH$ is an elliptic modular form of weight $k$ we can identify it as a particular Eisenstein series of level $N$ and deduce that $$C_F(0_2)= -\frac{B_k}{2k}\prod_{p \mid N} (1-p^{k-1}) \alpha_F^*(0).$$
This means that the $0_2$-th Fourier Jacobi coefficient coincides with $V_0(\varphi_1)$.

Using the Taylor expansion of the exponential we also get:
\be \begin{split} e\left[\tr T\bmat \tau & z \\ w & \tau^*\emat \right] =& e\left[ w\frac{\alpha}{i \sqrt{D_K}}\right]e\left[ z\frac{\ov{\alpha}}{-i \sqrt{D_K}}\right]e[n\tau + m\tau^*]\\
=& \sum_{\nu_1=0}^{\infty} \sum_{\nu_2=0}^{\infty} \left(\frac{\alpha}{i \sqrt{D_K}}\right)^{\nu_1}\left(\frac{\ov{\alpha}}{-i \sqrt{D_K}}\right)^{\nu_2}(2\pi i w)^{\nu_1}(2\pi i z)^{\nu_2}\frac{1}{\nu_1! \nu_2!} e[n\tau + m\tau^*]\end{split}\ee
Define $$A_{\nu_1, \nu_2}(F: n, m)=\sum_{\substack{\alpha \in \Oo\\ |\alpha|^2 \leq D_Knm}}\sum_{\substack{d \in \bfZ_+\\ d\mid n, d\mid m, d\mid \alpha\\\textup{gcd}(d,N)=1}} d^{k-1} \alpha^*_F\left(\frac{D_Knm-|\alpha|^2}{d^2}\right) \left(\frac{\alpha}{i \sqrt{D_K}}\right)^{\nu_1}\left(\frac{\ov{\alpha}}{-i \sqrt{D_K}}\right)^{\nu_2}.$$ Note that as in \cite{Kojima} one has that \begin{multline}A_{\nu_1, \nu_2}(F: n, m) = \sum_{\substack{\alpha \in \Oo\\ |\alpha|^2 \leq D_Knm}}\sum_{\substack{d \in \bfZ_+\\ d\mid n, d\mid m, d\mid \alpha\\\textup{gcd}(d,N)=1}} d^{k-1} \alpha^*_F\left(D_K \det\left(\bmat 1 & \ov{\alpha}/(-i \sqrt{D_K}d)\\ \alpha/(i \sqrt{D_K}d) & mn/d^2\emat\right)\right)\\ \times \left(\frac{\alpha}{i \sqrt{D_K}}\right)^{\nu_1}\left(\frac{\ov{\alpha}}{-i \sqrt{D_K}}\right)^{\nu_2}.\end{multline}  

On the other hand for $s \in \bfZ_+$ with $s$ dividing 
$\textup{gcd}(n,m,N(\alpha))$ one has 
\be\begin{split} \label{24} A_{\nu_1, \nu_2}(F: 1, mn/s^2) =&\sum_{\substack{\alpha \in \Oo\\ |\alpha|^2 \leq D_Knm/s^2}}\alpha^*_F\left(\frac{D_Knm-|\alpha|^2}{s^2}\right)\\ \times & \left(\frac{\alpha}{i \sqrt{D_K}}\right)^{\nu_1}\left(\frac{\ov{\alpha}}{-i \sqrt{D_K}}\right)^{\nu_2}\\= &A_{\nu_1, \nu_2}(F: mn/s^2,1) .\end{split}\ee
Hence we get $$A_{\nu_1, \nu_2}(F: n, m) = \sum_{\substack{s \in \bfZ_+\\ s\mid n, s\mid m\\\textup{gcd}(s,N)=1}}s^{k+\nu_1+\nu_2-1} A_{\nu_1, \nu_2}(F: 1, nm/s^2).$$ 
Finally using the function $A_{\nu_1,\nu_2}$
we can write the Fourier expansion of $F$ as \be \begin{split} F\left(\bmat \tau & z \\ w & \tau^*\emat \right) =& \sum_{m=1}^{\infty} \varphi_m (\tau, z,w) e[m\tau^*]\\
=& \sum_{m=1}^{\infty}\sum_{n=1}^{\infty} \sum_{\nu_1=0}^{\infty} \sum_{\nu_2=0}^{\infty} A_{\nu_1, \nu_2}(F: n,m)(2\pi i z)^{\nu_1}(2\pi i w)^{\nu_2}\frac{1}{\nu_1!\nu_2!}e[n\tau]e[m\tau^*].\end{split}\ee Hence in particular $$\varphi_m(\tau, z,w) = \sum_{n=1}^{\infty} \sum_{\nu_1=0}^{\infty} \sum_{\nu_2=0}^{\infty} A_{\nu_1, \nu_2}(F: n,m)(2\pi i z)^{\nu_1}(2\pi i w)^{\nu_2}\frac{1}{\nu_1!\nu_2!}e[n\tau].$$ Using this, let us compute $V_m \varphi_1$. To do so, we note that a full set of coset representatives of $\Gamma_0(N)\setminus \Delta_N(m)$ can be taken to be $$\left\{ \bmat A&B \\ &D \emat \mid A,B,D \in \bfZ_+, AD=m, 0 \leq B < D, \textup{gcd}(A,N)=1 \right\}.$$ 
We have 
\be \begin{split} (V_m \varphi_1)(\tau, z,w) = & m^{k-1} \sum_ {g \in \Gamma_0(N)\setminus \Delta_N(m)} (\varphi_1|_{k,1}g)(\tau, z,w)\\
=& m^{k-1} \sum_{\substack{AD=m\\ 0 \leq B < D\\ \textup{gcd}(A,N)=1}}D^{-k}\varphi_1\left(\frac{A\tau+B}{D}, \frac{\sqrt{m}z}{D}, \frac{\sqrt{m}w}{D}\right)D^{-k}\\
=& m^{k-1} \sum_{\substack{AD=m\\ 0 \leq B < D\\ \textup{gcd}(A,N)=1}}D^{-k}\sum_{n=1}^{\infty} \sum_{\nu_1=0}^{\infty} \sum_{\nu_2=0}^{\infty}A_{\nu_1, \nu_2}(F: n,1)\\
&\times \frac{\left(2\pi i \frac{\sqrt{m}z}{D}\right)^{\nu_1}\left(2\pi i \frac{\sqrt{m}w}{D}\right)^{\nu_2}}{\nu_1!\nu_2!}e\left[\frac{nA\tau+nB}{D}\right].
\end{split}\ee
Changing the order of summation and using the fact that for a fixed $D$ and $n$ one has $$\sum_{0 \leq B <D} e\left[\frac{nB}{D}\right] = \begin{cases} 0 \quad \textup{if $D \nmid n$} \\ D \quad\textup{if $D\mid n$}\end{cases}$$  we can re-write the above as
\be \begin{split} (V_m \varphi_1)(\tau, z,w) = &m^{k-1}\sum_{n=1}^{\infty} \sum_{\substack{AD=m\\  \textup{gcd}(A,N)=1\\ D \mid n}}D D^{-k}\sum_{\nu_1=0}^{\infty} \sum_{\nu_2=0}^{\infty}A_{\nu_1, \nu_2}(F: n,1)\\
&\times \frac{\left(2\pi i \frac{\sqrt{m}z}{D}\right)^{\nu_1}\left(2\pi i \frac{\sqrt{m}w}{D}\right)^{\nu_2}}{\nu_1!\nu_2!}e\left[\frac{nA\tau}{D}\right].
\end{split}\ee
Note that 
for a fixed $m$ one has  $$\sum_{n=1}^{\infty} \sum_{\substack{AD=m\\  \textup{gcd}(A,N)=1\\ D \mid n}} = \sum_{\substack{D\mid m\\ \textup{gcd}(m/D,N)=1}}\sum_{\substack{n=Dn'\\ n' \in \bfZ_+}}$$
This change of summation gives us (we still keep $A$ which is now defined to be $m/D$):
\be \label{Vmphi1}\begin{split} (V_m \varphi_1)(\tau, z,w) = &m^{k-1} \sum_{\substack{D\mid m\\ \textup{gcd}(A,N)=1}}D^{1-k}\sum_{n'=1}^{\infty} \sum_{\nu_1=0}^{\infty} \sum_{\nu_2=0}^{\infty}A_{\nu_1, \nu_2}(F: Dn',1)\\
&\times \frac{\left(2\pi i \frac{\sqrt{m}z}{D}\right)^{\nu_1}\left(2\pi i \frac{\sqrt{m}w}{D}\right)^{\nu_2}}{\nu_1!\nu_2!}e\left[n'A\tau\right]\\
=& \sum_{\substack{A\mid m\\ \textup{gcd}(A,N)=1}}A^{k-1}\sum_{n'=1}^{\infty} \sum_{\nu_1=0}^{\infty} \sum_{\nu_2=0}^{\infty}A_{\nu_1, \nu_2}(F: Dn',1)\\
&\times \frac{\left(2\pi i \frac{\sqrt{m}z}{D}\right)^{\nu_1}\left(2\pi i \frac{\sqrt{m}w}{D}\right)^{\nu_2}}{\nu_1!\nu_2!}e\left[n'A\tau\right],
\end{split}\ee
where in the last equality we combined $m^{k-1}$ with $D^{1-k}$ to yield $A^{k-1}$ and noted that we can as well sum over $A$ now defining $D:=m/A$. This for $N=1$ recovers precisely 
 the first equality in (3.6) in \cite{Kojima}, where 
Kojima's $d$ is $m/a$.  

To compare with $V_m \varphi_1$ we now calculate 
\be \begin{split} \varphi_m\left(\tau, \frac{z}{\sqrt{m}},\frac{w}{\sqrt{m}}\right) =& \sum_{n=1}^{\infty} \sum_{\nu_1=0}^{\infty} \sum_{\nu_2=0}^{\infty}A_{\nu_1, \nu_2}(F: n,m) \frac{\left(2\pi i \frac{z}{\sqrt{m}}\right)^{\nu_1} \left(2\pi i \frac{w}{\sqrt{m}}\right)^{\nu_2}}{\nu_1!\nu_2!}e[n\tau]\\
=& \sum_{n=1}^{\infty} \sum_{\nu_1=0}^{\infty} \sum_{\nu_2=0}^{\infty}\sum_{\substack{s \in \bfZ_+\\ s \mid n, s\mid m\\\textup{gcd}(s,N)=1}} s^{k+\nu_1+\nu_2-1}A\left(F: \frac{mn}{s^2}, 1\right)  \\
\times & \frac{\left(2\pi i \frac{z}{\sqrt{m}}\right)^{\nu_1} \left(2\pi i \frac{w}{\sqrt{m}}\right)^{\nu_2}}{\nu_1!\nu_2!}e[n\tau].
\end{split}\ee
Using, similarly as before, that for a fixed $m$ one has $$\sum_{n=1}^{\infty} \sum_{\substack{s \in \bfZ_+\\ s \mid n, s\mid m}} = \sum_{\substack{s \in \bfZ_+\\  s\mid m}}\sum_{\substack{n=sn'\\ n'\in \bfZ_+}}$$ we get that 
\be \begin{split} \varphi_m\left(\tau, \frac{z}{\sqrt{m}},\frac{w}{\sqrt{m}}\right) =& \sum_{\nu_1=0}^{\infty} \sum_{\nu_2=0}^{\infty}\sum_{\substack{s \in \bfZ_+\\ s\mid m\\\textup{gcd}(s,N)=1}} \sum_{n'=1}^{\infty} s^{k+\nu_1+\nu_2-1}A_{\nu_1, \nu_2}\left(F: \frac{mn'}{s}, 1\right)  \\
\times & \frac{\left(2\pi i \frac{z}{\sqrt{m}}\right)^{\nu_1} \left(2\pi i \frac{w}{\sqrt{m}}\right)^{\nu_2}}{\nu_1!\nu_2!}e[n's\tau]\\
=&\sum_{\nu_1=0}^{\infty} \sum_{\nu_2=0}^{\infty}\sum_{\substack{s \in \bfZ_+\\ s\mid m\\\textup{gcd}(s,N)=1}} \sum_{n'=1}^{\infty} s^{k+\nu_1+\nu_2-1}A_{\nu_1, \nu_2}\left(F: \frac{mn'}{s}, 1\right)  \\
\times & \frac{\left(2\pi i \frac{\sqrt{m}z}{(m/s)}\right)^{\nu_1} \left(2\pi i \frac{\sqrt{m}w}{(m/s)}\right)^{\nu_2}}{\nu_1!\nu_2!}s^{-\nu_1-\nu_2}e[n's\tau].
\end{split}\ee
Now redefining $A:=s$ and $D=m/s$, we get 
\be \begin{split} \varphi_m\left(\tau, \frac{z}{\sqrt{m}},\frac{w}{\sqrt{m}}\right) =&\sum_{\nu_1=0}^{\infty} \sum_{\nu_2=0}^{\infty}\sum_{\substack{A \in \bfZ_+\\ A\mid m\\\textup{gcd}(A,N)=1}} \sum_{n'=1}^{\infty} A^{k-1}A_{\nu_1, \nu_2}\left(F: Dn', 1\right)  \\
\times & \frac{\left(2\pi i \frac{\sqrt{m}z}{D}\right)^{\nu_1} \left(2\pi i \frac{\sqrt{m}w}{D}\right)^{\nu_2}}{\nu_1!\nu_2!}e[n'A\tau].
\end{split}\ee

Note that this is the same as (\ref{Vmphi1}). Hence we have proved that $$\varphi_m\left(\tau, \frac{z}{\sqrt{m}}, \frac{w}{\sqrt{m}}\right) = (V_m \varphi_1)(\tau, z,w)$$ and so $\varphi_1$ indeed determines $\varphi_m$ for all $m$, and thus Proposition \ref{unique} is proved.
\end{proof}

\begin{definition}
Let  $J_{k,1}^{\rm spez}(N)$ be the subspace of $J_{k,1}(N)$ of \emph{special} Jacobi forms $\varphi$ whose Fourier coefficients 
 $c_ {\varphi}(\ell, t)$ only depend on $\ell - N(t)$.
\end{definition}

Our goal now is to prove that the map from $\mM_k^*(N)$ to $J_{k,1}(N)$ defined  in the proof of Proposition \ref{unique} is  an isomorphism to $J_{k,1}^{\rm spez}(N)$.
By generalizing Propositions 1.3 and 1.4 of \cite{Haverkamp95} to the case $N\geq 1$ we know:

\begin{prop}
For $D_K=4, 8$ or for $D_K \equiv 3$ (mod 4)  prime we have $$J_{k,1}^{\rm spez}(N)=J_{k,1}(N).$$ (For other $D_K$ these are not equal.)
\end{prop}

The following theorem generalizes part of \cite{Haverkamp95} Satz 7.6 (who treats $N=1$).
\begin{thm} \label{Jacobiisom}
We have an isomorphism $J_{k,1}^{\rm spez}(N) \cong \mathcal{M}^*_k(N)$.
\end{thm}

\begin{proof}
The proof of Proposition \ref{unique} shows that we have an injection of the right hand side into the left hand side, i.e. a form in the Maass space is determined by its first Fourier-Jacobi coefficient.

Given $\phi \in J_{k,1}^{\rm spez}(N)$, 
we claim that \be \label{formjac} I(\phi)\left(\bmat  \tau&z\\w &\tau^* \emat  \right):=\sum_{m=0}^{\infty} (V_m \phi)(\tau,z \sqrt{m}, w \sqrt{m}) e[m \tau^*] \in \mathcal{M}_k^*(N).\ee (This is also stated in \cite{TY12} Proposition 2.2 but we decided to give a proof following \cite{Kojima} for the convenience of the reader.)
For convergence of this series we argue as \cite{Haverkamp95} Satz 7.2 (or \cite{Ibukiyama12} Section 3.1). 
We claim that $I(\phi)$ has a Fourier expansion with coefficients $$C_{I(\phi)}(0_2)=-c_{\varphi}(0,0)   \frac{B_k}{2k}\prod_{p \mid N} (1-p^{k-1})$$ and \begin{equation} \label{FJcoeff}C_{I(\phi)}(T)=\sum_{\substack{d \in \bfZ_{+}, d \mid \epsilon(T)\\ \textup{gcd}(d,N)=1}} d^{k-1} c_{\phi}(\frac{\ell m}{d^2}, \frac{t}{d}) \, \text{ for } T=\bmat  \ell&t\\\ov{t}&m\emat \geq 0, T \neq 0_2.\end{equation}

For this we follow the proof of 
 \cite{EichlerZagier85}  Theorem 4.2 (7) and the explicit expression for $V_m$ given in the proof of Proposition \ref{unique}:
For $m \geq 1$ we have \begin{eqnarray*} (V_m \phi)(\tau, z \sqrt{m}, w \sqrt{m})&=&m^{k-1} \sum_{\substack{ad=m\\\textup{gcd}(a,N)=1}} \sum_{b {\hspace{1pt} \rm mod \hspace{1pt}}d} d^{-k} \phi(\frac{a \tau+b}{d}, \frac{m z}{d}, \frac{m w}{d})\\&=&m^{k-1} \sum_{\substack{ad=m\\ \textup{gcd}(a,N)=1}} \sum_{b {\hspace{1pt} \rm mod \hspace{1pt}}{d}} d^{-k} \sum_{\substack{\ell \in \mathbf{Z}_{\geq 0}, t \in \mD^{-1}\\ N(t) \leq \ell}} c_{\phi}(\ell, t) e[\frac{\ell a \tau}{d} + \ov{t} \frac{mz}{d} + t \frac{mw}{d}]e[\frac{\ell b}{d}]\\ 
&=&m^{k-1} \sum_{\substack{ad=m\\\textup{gcd}(a,N)=1}}d^{1-k} \sum_{\substack{\ell \in \mathbf{Z}_{\geq 0}, d \mid \ell\\ t \in \mD^{-1}, N(t) \leq \ell}} c_{\phi}(\ell, t) e[\frac{\ell a \tau}{d} + \ov{t} \frac{mz}{d} + t \frac{mw}{d}].\end{eqnarray*} Now we eliminate all $d$'s and  change variables by writing $\ell$ for $\ell/d$ to get that the above equals
$$\sum_{\substack{a \mid m\\ \textup{gcd}(a,N)=1}} a^{k-1} \sum_{\substack{\ell \in \mathbf{Z}_{\geq 0}, t \in \mD^{-1}\\ N(t) \leq \ell m/a}} c_{\phi}(\frac{\ell m}{a}, t) e[\ell a \tau + \ov{t} az + t aw].$$ Changing the order of summation and again changing variables by writing $\ell$ for $a \ell$ and $t$ for $at$ we now get
$$\sum_{\substack{\ell \in \mathbf{Z}_{\geq 0}, t \in \mD^{-1}\\ N(t) \leq \ell m}} \sum_{\substack{a \mid \epsilon(T)\\ \textup{gcd}(a,N)=1}} a^{k-1} c_{\phi}(\frac{\ell m}{a^2}, \frac{t}{a}) e[\ell \tau + \ov{t} z + t w].$$

So, if we assume for the moment that $I(\phi) \in \mM_k(N)$ 
this expression for the Fourier expansion implies that $I(\phi)$  lies in the Maass space with the function $\alpha^*_{I(\phi)}$ describing its Fourier coefficients given by \be \label{desc} \alpha^*_{I(\phi)}(D_K {\rm det} T/d^2)= c_{\phi}(\frac{\ell m}{d^2}, \frac{t}{d}).\ee
The latter is well-defined since we assumed $\phi \in J_{k,1}^{\rm spez}(N)$. 

It now remains to prove that $I(\phi) \in \mM_k(N)$. We first adapt the proof of \cite{Klingen59} Satz 3 (in the Siegel case) and that of \cite{Grosche78} Satz 3 (principal congruence subgroup in the Hermitian case) to prove the following:

\begin{prop} \label{generation}
$\Gamma_0^{(2)}(N)$ is generated by matrices of the form \begin{itemize} \item $\bmat  {}^t\ov{U}&0\\0&U^{-1}\emat $ for $U \in {\rm GL}_2(\Oo)$, 
and \item $\bmat  a&0&b&0\\0&1&0&0\\c&0&d&0\\0&0&0&1\emat $ for $\bmat  a&b\\c&d\emat  \in \Gamma_0(N) \subset {\rm SL}_2(\bfZ)$.\end{itemize}
\end{prop}
\begin{proof}
For a matrix $S\in \mS_2(\bfZ)$ 
we will write $I(S):= \bmat I_2 & S \\ 0_2 & I_2\emat$ (translation by $S$) and $I'(S) := \bmat I_2&0_2\\ S & I_2\emat$ (antitranslation by $S$). Also, for $U \in \GL_2(\Oo)$ we set $R(U):= \bmat {}^t\ov{U} & 0 \\ 0& U^{-1}\emat$ (rotation by $U$). We need the following lemmas.

\begin{lemma} \label{l3} Let $\bmat a_1\\a_2\\c_1\\c_2\emat$ be the first column of an element of $\Gamma_0^{(2)}(N)$.  Then there exists   $S\in \mS_2(\bfZ)$  such that 
$$\bmat a_1'\\a_2'\\c_1'\\c_2'\emat := I(S) \bmat a_1\\a_2\\c_1\\c_2\emat$$ satisfies $(a_1', a_2'):=a_1'\Oo + a_2'\Oo=\Oo$.
 \end{lemma}

\begin{proof} Apply \cite{Grosche78} Korollar on p. 149 with $n=2$ and $\fq=\Oo$. \end{proof}
\begin{lemma} \label{l2} Let $\fa=\bmat a_1 \\ a_2\emat$ be a column vector with $a_1, a_2 \in \Oo$ satisfying $(a_1, a_2) =\Oo$. Then there exists $U \in \SL_2(\Oo)$ such that $U \fa = \bmat 1 \\ 0 \emat.$ \end{lemma} 
\begin{proof} Let $x,y \in \Oo$ be such that $xa_1 + ya_2=1$. Then set $U=\bmat x&y \\ -a_2& a_1\emat$. \end{proof}

Let us now prove Proposition \ref{generation}. 
Let $M \in \Gamma_0^{(2)}(N)$. By Lemma \ref{l3} we can find $S \in \mS_2(\bfZ)$ such that $I(S) M$ has the first column whose top two entries $a_1, a_2$ satisfy  $(a_1, a_2)=\Oo$.
So by Lemma \ref{l2} we can multiply $M$ on the left by some $R(U)$ to ensure that the first column of $M$ starts of with $\bmat 1 \\ 0 \emat$. Furthermore, since $M=\bmat A&B\\ C&D\emat \in \U(2,2)(\bfZ)$ we must have ${}^t\ov{A} C \in \mS_2(\bfZ)$. For $A$ of the form $\bmat 1&*\\ 0&*\emat$ this forces the top left entry of $C$ to be in $\bfZ$.  Hence there exists $S \in \mS_2(\bfZ)$ such that $I'(S)M$  has the first column of the form $\bmat 1\\0\\0\\0 \emat$. Moreover, since the lower-left block of $M$ has all entries divisible by $N$ we can take $S$ with $S\equiv 0$ (mod $N$). 
So, now we have $$M=\bmat 1&*&*&*\\ 0&a_0 & * & b_0\\ 0&0&*&*\\0&c_0&*&d_0\emat$$ with $\bmat a_0& b_0 \\ c_0 & d_0\emat \in \Gamma_0^{(1)}(N)$ (note that the zero above $c_0$ is a consequence of ${}^t\ov{A}C$ being hermitian).

It follows from Lemma \ref{Hilfsatz2} and the above arguments that  
 the group $\Gamma_0^{(2)}(N)$ is generated by $I(S), I'(S), R(U)$ and the subset $\fD\subset \Gamma_0^{(2)}(N)$ of matrices of the form $\bmat 1&*&*&*\\ 0&a_0 & * & b_0\\ 0&0&*&*\\0&c_0&*&d_0\emat$ with $\bmat a_0& b_0 \\ c_0 & d_0\emat \in \Gamma_0(N)$. 
 In fact, we only need to consider the set $\fD_0$ of matrices $M=\bmat 1&*&*&*\\ 0&a & * & b\\ 0&0&1&*\\0&c&*&d\emat$ with all $*$ equal to $0$, as we can multiply on the left by $\bmat 1&0&0&0\\ 0&d & 0 & -b\\ 0&0&1&0\\0&-c&0&a\emat$ 
to get $M$ to be of the form $I(S)R(U)$ for appropriate $S, U$. Note that for any matrix of the form $\bmat 1&0&0&0\\ 0&a & 0 & b\\ 0&0&x&0\\0&c&0&d\emat$ we must have $x=1$ since ${}^t\ov{A}D-{}^t\ov{C}B=I_2$. 

 Now we show that we also do not need the translations $I(S)$. It is enough to show that we do not need $I(S)$ for $S\in \mA:= \left\{\bmat 1&0\\ 0&0\emat, \bmat 0&0\\ 0&1\emat, \bmat 0&1\\ 1&0\emat, \bmat 0&\omega\\ \ov{\omega}&0\emat\right\}$ for $$\omega=\begin{cases}\sqrt{-D}& \text{if} -D \not \equiv 1 \mod{4}\\ \frac{1}{2}(1+\sqrt{-D})& \text{if} -D \equiv 1 \mod{4}, \end{cases}$$ as the rest of  matrices in $\mS_2(\bfZ)$ are a $\bfZ$-linear combination of elements of $\mA$. 

Since $I(\bmat 0&0\\ 0&1\emat) \in \fD_0$ we clearly do not need $\bmat 0&0\\ 0&1\emat$. By using (7) in \cite{Klingen59} with $U$ a permutation matrix we see that $I( \bmat 1&0\\ 0&0\emat) = R(U) I(\bmat 0&0\\ 0&1\emat) R(U)^{-1}$,
 so we also do not need the first matrix in $\mA$. We have $$\bmat 1&1\\ 0&1 \emat \bmat 0 & 0\\ 0&1\emat \bmat 1&0\\ 1&1\emat = \bmat 0&1\\ 1&0 \emat + \bmat 1&0\\ 0&1 \emat,$$ so we can again use (7) in \cite{Klingen59} to see that the third matrix in $\mA$ is also not needed, because this gives us $$R(\bmat 1&1\\ 0&1 \emat) I(\bmat 0 & 0\\ 0&1\emat) R(\bmat 1&0\\ 1&1\emat) = I( \bmat 0&1\\ 1&0 \emat)I( \bmat 1&0\\ 0&1 \emat),$$ and we already know that the left hand side is generated by rotations and $\fD_0$ and also the last term on the right-hand side is. Finally, one has 
$$\bmat 1& \omega\\0&1\emat \bmat 0&0\\0&1\emat  \bmat 1&0\\ \ov{\omega}&1\emat  = \bmat 0& \omega\\ \ov{\omega} & 0 \emat  + \bmat \omega\ov{\omega}&0\\0&0\emat + \bmat 0&0\\0&1\emat,$$ hence by the same argument as above we see that the fourth matrix in $\mA$ is also not needed. We can apply the same arguments to antitranslations (replacing $\mA$ by $N\mA$).

 Lastly note that we can replace $\fD_0$ with the second set of matrices in the statement of Proposition \ref{generation} because $$\tilde{J} \bmat a&0&b&0\\0&1&0&0\\c&0&d&0\\0&0&0&1\emat \tilde{J}^{-1} = \bmat 1&0&0&0\\ 0&a & 0 & b\\ 0&0&1&0\\0&c&0&d\emat,$$ 
 where
$\tilde{J}=\bmat 0&1&0&0\\ 1&0&0&0\\ 0&0&0&1\\ 0&0&1&0\emat\in \Gamma^{(2)}_0(N)$.
\end{proof}

Define $\Gamma_{2,1}(N)$ to be all the elements of $\Gamma_0^{(2)}(N)$ whose last row is $(0,0,0,1)$.
Since $V_m \phi \in J_{k,m}(N)$ we get (as in \cite{Haverkamp95} Lemma 1.2 and (22)) that $V_m \phi|_k [M]$ 
for all $M \in \Gamma_{2,1}(N)$. By considering the Fourier expansion of $I(\phi)$ we further see that $I(\phi)$ is invariant under $\bmat J&0_2\\0_2&J \emat $ for $J=\bmat  0&-1\\1&0\emat $.

This allows us to deduce as in \cite{Haverkamp95} Satz 7.2 that $I(\phi)$ is invariant under $\epsilon I_4$ for $\epsilon \in \Oo^{\times}$,  
$\bmat J&0_2\\0_2&J \emat $  and $\Gamma_{2,1}(N)$, which in particular includes $R(U)$ for $U=\bmat \epsilon&\lambda \\ 0&1\emat$.  As \cite{Haverkamp95} shows in the proof of Satz 7.2 matrices of the form  $\bmat  {}^t\ov{U}&0_2\\0_2&U^{-1}\emat $ for $U \in {\rm GL}_2(\Oo)$ can be generated by these,
 so we get by Proposition \ref{generation} that $I(\phi)$ is invariant under $\Gamma_0^{(2)}(N)$ 
and therefore $I(\phi)\in  \mathcal{M}_k(N)$ as desired. 
\end{proof}

\section{Some transformation properties of the theta function}\label{Some transformation properties of the theta function}
In this section we discuss the relationship between Jacobi forms of odd  level $N$ and elliptic modular forms, which uses the so-called theta decomposition of Jacobi forms. For later use we prove a result about the transformation property of the theta functions occurring in this decomposition.
To shorten notation in this section we will write $D$ for $D_K$.

\subsection{Theta decomposition} \label{Theta decomposition}
For $u\in \mD^{-1}$ define $$\vartheta_u(\tau, z, w):=\sum_{a \in u+ \Oo} e[N(a) \tau+\ov{a}z+aw].$$ 
Consider a Jacobi form $\varphi \in J^{\rm spez}_{k,1}(N)$. 
Then its Fourier expansion can be written as \be \label{decompJ} \varphi(\tau, z, w)=\sum_{u \in \mD^{-1}/\Oo} f_u(\tau) \vartheta_u(\tau, z, w), \ee where \be \label{fut} f_u(\tau):=\sum_{\substack{\ell \geq 0\\ \ell \equiv -DN(u) \hspace{1pt} \textup{mod}\hspace{1pt}D}} \alpha_{\varphi}^*(\ell) e[\ell \tau/D]\ee
 for $\alpha_{\varphi}^*(D(m-N(u))):=c_{\varphi}(m,u)$. 
The latter is well-defined by the definition of $ J^{\rm spez}_{k,1}(N)$ and the decomposition \eqref{decompJ} is unique, since the $\vartheta_u$ are linearly independent as functions $(z,w) \mapsto \vartheta_u(\tau, z, w)$ for fixed $\tau \in \bfH$ (see e.g. \cite{Haverkamp95} Proposition 5.1).

\begin{lemma} [\cite{Kojima} Lemma 2.1 (for $\bfQ(i)$),\cite{Shintani75} Proposition 1.6] \label{trafo1}
For $\sigma=\bmat  a&b\\c&d \emat  \in {\rm SL}_2(\bfZ)$ and $u \in \mD^{-1}$ we have
$$\vartheta_u|[\sigma]_1=\sum_{v \in \mD^{-1}/\Oo} M_{u,v}(\sigma) \vartheta_v,$$
 where
$$M_{u,v}(\sigma)=\begin{cases} \frac{-i}{c\sqrt{D}} \sum_{\gamma \in u+ \Oo/c\Oo} e\left[ \frac{a |\gamma|^2-\gamma \ov{v}- \ov{\gamma}v + d |v|^2}{c} \right] & \text{ if } c \neq 0\\ {\rm sign}(a) \delta_{u,av} e[ab|u|^2] & \text{ if } c=0.\end{cases}$$
If $c>0$ and $D \mid c$ then $$M_{u,v}(\sigma)=\delta_{u,dv} e[ab |u|^2] \chi_K(|d|).$$ 
\end{lemma}

 Fix an ordering on
$\mD^{-1}/\Oo$, a group with $D$ elements. For $\sigma \in \SL_2(\bfZ)$ let $M(\sigma)$ be the $D \times D$ matrix whose $(u,v)$-entry is $M_{u,v}(\sigma)$ as defined above. Since $\varphi$ is of level $N$,
 it follows from Lemma \ref{trafo1} and from \cite{Kojima} Lemma 2.2 (which is easily  reproven
for level greater than one) 
that for $\sigma \in \Gamma_0(N)$ we have
$$f_u|_{k-1}\sigma=\sum_{v \in \mD^{-1}/\Oo} N_{u,v}(\sigma) f_v,$$
 for the matrix $N(\sigma)=(N_{u,v}(\sigma))$ defined by $${}^tN(\sigma) M(\sigma)=I_D.$$ Here we use the notation $$(f|_{k-1} \sigma)(z)=(cz+d)^{1-k} f (\sigma(z)) \text{ for } \sigma=\bmat *&*\\c&d \emat \in \GL_2(\bfR)^+.$$ 

From this we can deduce the following: 

\begin{lemma}[Analogue of Corollary in section 4 of \cite{Krieg91} and Korollar 4.4 and Satz 4.5 in \cite{Haverkamp95}] \label{trafo}
For $u \in \mD^{-1}$ we have
\begin{enumerate}
\item $f_u|_{k-1}M=f_u, \text{ if } M \equiv I_2 \mod{DN}$,
\item $f_u(\tau+1)=e[-N(u)] f_u(\tau)$,
\item $\displaystyle f_0|_{k-1} \bmat  1&0\\N&1 \emat =\frac{1}{D} 
\sum_{v,u \in \mD^{-1}/\Oo}  e[N N(v)+u \ov{v}+\ov{u}v] f_u,$
\item $f_0|_{k-1} M=\chi_K(d) f_0 \text{ for } M=\bmat  a&b\\c&d\emat  \in \Gamma_0(ND)$.
\end{enumerate}
\end{lemma}

\begin{proof}
The proofs are similar to those in \cite{Krieg91} and \cite{Haverkamp95}. We sketch the proof of (iii) since it is the hardest.
 Note first that $$ \bmat  1&0\\N&1 \emat =J  \bmat  1&-N\\0&1 \emat  J^{-1},$$ where  $J= \bmat  0&-1\\1&0 \emat $.
Now by 
Lemma \ref{trafo1}
 we have $$\vartheta_0|_{1,1}[J]=\frac{-i}{\sqrt{D}} \sum_{v \in \mD^{-1}/\Oo}  \vartheta_v$$ 
and $$\vartheta_v|[\bmat  1&-N\\0&1 \emat ]_{1}=e[-N N(v)] \vartheta_v.$$ 
 Note that $J^{-1}=-J$ and by Lemma \ref{trafo1}
the action of $\sigma=-I_2 \in {\rm SL}_2(\bfZ)$ is given by $M_{h,k}(\sigma)=-\delta_{h,-k}$.
Putting this together we obtain
\begin{eqnarray*}\vartheta_0|[ \bmat  1&0\\N&1 \emat ]_{1}&=&\frac{i}{\sqrt{D}} \sum_{v\in \mD^{-1}/\Oo} e[-NN(v)] \vartheta_v|[J]_{1}\\ &=&\frac{i}{\sqrt{D}} \sum_{v\in \mD^{-1}/\Oo} e[-NN(v)] \cdot \left( \frac{-i}{\sqrt{D}} \sum_{u \in \mD^{-1}/\Oo} e[-u\ov{v}-\ov{u}v] \vartheta_u \right)\\&=&\frac{1}{D} 
\sum\limits_{v,u \in \mD^{-1}/\Oo}  e[-N N(v)+u \ov{v}+\ov{u}v] \vartheta_u.\end{eqnarray*} 
\end{proof}

We recall the following result about quadratic Gauss sums:
\begin{lemma}[\cite{Haverkamp95} Lemma 0.4] \label{gauss}
Assume $D$ is odd and coprime to an odd integer $N$. Let $a \in \bfZ$ be coprime to $N$. Then
$$\frac{1}{N}\sum_{\gamma \in  \Oo/N\Oo} e\left[ a \frac{|\gamma|^2}{N}\right]=\chi_K(N).$$
\end{lemma}

To relate the Jacobi modular forms to elliptic modular forms we make the following definition:
\begin{definition} \label{def3.4}
For $N \in \bfZ_{\geq 0}$ let $$M^+_{k-1}(DN,\chi_K):= \{f =\sum_n a_n(f) q^n \in M_{k-1}(DN, \chi_K) \mid a_n=0 \hs \textup{whenever} \hs \chi_K(n)=1\}.$$
\end{definition}

At  this point we can now generalize 
\cite{Krieg91} Proposition 4 and \cite{Haverkamp95} Proposition 5.6 to general odd level $N$:

Assume that $\textup{gcd}(D,N)=1$. Then there exist $w,y \in \bfZ$ such that $Dw-Ny=1$, and we define $$W_D:=\bmat  D&y\\DN&Dw\emat =\bmat  1&0\\N&1 \emat \bmat  D&y\\0&1 \emat .$$

\begin{prop} \label{prop3.5}
Let $D$ 
be prime. 
Assume $N$ is odd. Then the mapping $J_{k,1}(N) \to M_{k-1}^+(DN,\chi_K)$ given by
$$\varphi \mapsto f:=f_0|_{k-1}W_D$$
is an injective homomorphism. The Fourier coefficients of $f$ satisfy \begin{equation} \label{Fcoeff2} a_{\ell}(f)=i \frac{a_D(\ell)}{\sqrt{D}} \chi_K(N)  \alpha_{\varphi}^*(\ell),\end{equation} 
where $$a_D(\ell)=\#\{u \in \mD^{-1}/\Oo| DN(u) \equiv -\ell \mod{D}\}.$$ 
\end{prop}

\begin{rem}
For $N=1$ \cite{Krieg91} section 6 also proves surjectivity of the map. We show in the proof of Theorem \ref{lift} that for $N=p$ prime the map is surjective onto the space of $p$-old cuspforms in $M_{k-1}^+(Dp,\chi_K)$.

\end{rem}

\begin{proof} [Proof of Proposition \ref{prop3.5}]
 Using the argument from \cite{Haverkamp95} Satz 5.3 proving moderate growth and Lemma \ref{trafo} (and the fact that the involution induced by $W_D$ preserves $M_{k-1}(DN, \chi_K)$ by Proposition 1.1 in \cite{AtkinLi78}) we can conclude
 that $F \in M_{k-1}(DN, \chi_K)$.

We now calculate the Fourier expansion of $F$. Note that $$W_D:=\bmat  D&y\\DN&Dw\emat =\bmat  1&0\\N&1 \emat \bmat  D&y\\0&1 \emat =\bmat  1&0\\N&1 \emat \bmat  1&y\\0&1 \emat \bmat  D&0\\0&1 \emat .$$ We first need to work out the effect of $$\sigma=\bmat  1&0\\N&1 \emat \bmat  1&y\\0&1 \emat =\bmat  1&y\\N&Ny+1 \emat =\bmat  1&y\\N&Dw \emat $$ on $\vartheta_0$: For this we use similar ideas to those in Shintani's proof of \cite{Shintani75} Proposition 1.6(ii):
Since $\textup{gcd}(N,D)=1$ note that $v \mapsto vN$
 induces an automorphism of $\mD^{-1}/\Oo$. Applying this change of variable we get
$$M_{0,vN}(\sigma)=\frac{-i}{\sqrt{D}N} \sum_{\gamma \in  \Oo/N\Oo} e\left[ \frac{|\gamma|^2}{N}\right]e[-\gamma \ov{v}- \ov{\gamma}v + wND |v|^2].$$
Since $v \in \mD^{-1}$ we have $e[-\gamma \ov{v}- \ov{\gamma}v] =1$, 
and since $\sqrt{D} \mD^{-1} \subset \Oo$ we also have $e[wD |v|^2]=1$.
By Lemma \ref{gauss}
 we get $$\vartheta_0|[\sigma]_1=\chi_K(N) \frac{-i}{\sqrt{D}}\sum_{v \in \mD^{-1}/\Oo} \vartheta_v,$$ which implies that $$f_0|_{k-1}\sigma=\chi_K(N) \frac{i}{\sqrt{D}}\sum_{v \in \mD^{-1}/\Oo} f_v,$$ and so $$f(\tau)=f_0|_{k-1}W_D(\tau)=\chi_K(N) \frac{i}{\sqrt{D}}\sum_{v \in \mD^{-1}/\Oo} f_v(D\tau).$$ 
This implies the formula \eqref{Fcoeff2} for the Fourier coefficients of $f$.
As we know from \cite{Krieg91} 4 (5) on p. 670 that $a_D(\ell)=1+ \chi_K(-\ell)=1-\chi_K(\ell)$ 
we conclude from this that $f \in M_{k-1}^+(DN, \chi_K)$.

As $\varphi \mapsto f$ is a linear map of vector spaces it suffices for the injectivity to check that the kernel is trivial. 
 Suppose that $f=f_0|_{k-1}W_D=0$. By (\ref{Fcoeff2}) this means that  $\alpha_{\varphi}^*(\ell)=0$ whenever $a_D(\ell) \neq 0$. By the definition of $ a_D(\ell)$ these are the only $\alpha_{\varphi}^*(\ell)$ used in the definition of the $f_u$, so the Jacobi form $\varphi=0$. 
\end{proof}

\subsection{Key technical result in this section}

\begin{prop} \label{general}
Assume that $D$ is prime  (which implies $D \equiv -1 \mod{4}$), 
$p>2$ a split prime in $K/\bfQ$. 
Note 
 that this is equivalent to  $\chi_K(p)=
1.$ 
Let $\pi \in \Oo$ with $N(\pi)=p$. 
Then $$M_{\pi u,\pi v}(\sigma)=M_{u,v}(\bmat  p&0\\0&1\emat \sigma \bmat  p^{-1}&0\\0&1\emat ) $$ for all $\sigma \in \Gamma_0(p)$.
\end{prop}

\begin{proof}
Let $\sigma=\bmat  a&b\\pc&d\emat $. Note that $$\bmat  p&0\\0&1\emat \sigma \bmat  p^{-1}&0\\0&1\emat =\bmat  a&pb\\c&d\emat .$$
We consider three different cases:

\begin{enumerate}
\item $c=0$
\item $(c,D)=1$
\item $c>0$ and $D \mid c$
\end{enumerate}

Case (1.) is straightforward:
 Lemma \ref{trafo1} tells us that $$M_{\pi u, \pi v}(\sigma)={\rm sign}(a) \delta_{\pi u, \pi av} e[ab |\pi u|^2].$$ This is clearly equal to $M_{u,v}(\bmat  p&0\\0&1\emat \sigma \bmat  p^{-1}&0\\0&1\emat )={\rm sign}(a) \delta_{u, av} e[apb | u|^2]$.

For Case (2.) we calculate by 
Lemma \ref{trafo1} that \begin{eqnarray*} M_{\pi u, \pi v}(\sigma)&=&\frac{-i}{\sqrt{D}} \frac{1}{pc} \sum_{\gamma \in \pi u + \Oo/pc\Oo} e\left[\frac{1}{pc}(a |\gamma|^2-\gamma \ov{\pi v}-\ov{\gamma}\pi v+dp|v|^2)\right]\\
&=&\frac{-i}{\sqrt{D}} \frac{1}{pc} \sum_{\gamma \in \Oo/pc\Oo} e\left[\frac{1}{pc}(a |\gamma+\pi u|^2-(\gamma+\pi u) \ov{\pi v}-\ov{(\gamma+\pi u)}\pi v+dp|v|^2)\right]\\ 
&=&\frac{-i}{\sqrt{D}} \frac{1}{pc} \sum_{\gamma \in \Oo/pc\Oo} e\left[\frac{1}{pc}(a \left|\gamma+\pi u-\frac{\ov{\pi v}}{a}\right|^2+(d-\frac{1}{a})p|v|^2)\right]\\ 
&=&\frac{-i}{\sqrt{D}} e[\frac{1}{c}(d-\frac{1}{a})|v|^2)] \cdot \frac{1}{pc} \sum_{\gamma \in \Oo/pc\Oo} e\left[\frac{a}{pc} \left|\gamma+\left( \pi u-\frac{\ov{\pi v}}{a}\right)\right|^2\right].\end{eqnarray*}

On the other hand, a similar calculation shows that  $$M_{u,v}(\bmat  a&pb\\c&d\emat )=\frac{-i}{\sqrt{D}} e[\frac{1}{c}(d-\frac{1}{a})|v|^2)] \cdot \frac{1}{c} \sum_{\gamma \in \Oo/c\Oo} e\left[\frac{a}{c}\left|\gamma+\left( u-\frac{\ov{v}}{a}\right)\right|^2\right].$$

So these would be equal if we can show that \begin{equation} \label{claimgs} \frac{1}{pc} \sum_{\gamma \in \Oo/pc\Oo} e\left[\frac{a}{pc}\left|\gamma+\pi u-\frac{\ov{\pi v}}{a}\right|^2\right]=\frac{1}{c} \sum_{\gamma \in \Oo/c\Oo} e\left[\frac{a}{c}\left|\gamma+\left( u-\frac{\ov{v}}{a}\right)\right|^2\right]. \end{equation}

\begin{lemma}\label{37} For $D \equiv -1 \mod{4}$, $u, v \in \mD^{-1}$, $\textup{gcd}(N, aD)=1$ we have
$$\frac{1}{N} \sum_{\gamma \in \Oo/N\Oo} e\left[\frac{a}{N}\left|\gamma+\left( u-\frac{\ov{v}}{a}\right)\right|^2\right]=\frac{1}{N} \sum_{\gamma \in \Oo/N\Oo} e\left[\frac{a}{N}\left|\gamma\right|^2\right].$$
\end{lemma}

\begin{proof}
We first note that we can work modulo $N$ in the argument of $e\left[\frac{1}{N}\cdot\right]$.
 Since $\textup{gcd}(a,N)=1$ there exists $a^* \in \bfZ$ such that $a a^*\equiv 1 \mod{N}$, so $$\frac{\ov{v}}{a}= a a^* \frac{\ov{v}}{a} \equiv a^* \ov{v} \mod{N}.$$ It therefore suffices to prove the statement for a general element $u \in \mD^{-1}$ and $v=0$. We take the $\bfZ$-bases of $\Oo$ and $\mD^{-1}$ as follows:
 $$\Oo=\bfZ+\frac{1}{2}(1+\sqrt{-D}) \bfZ$$ and $$\mD^{-1}=\frac{i}{\sqrt{D}}\bfZ+\frac{1}{2}(1+\frac{i}{\sqrt{D}}) \bfZ.$$ Writing $u=u_1 \frac{i}{\sqrt{D}}+\frac{u_2}{2}(1+\frac{i}{\sqrt{D}})$ for $u_1, u_2 \in \bfZ$ we calculate that 
$$u\equiv u_1 \sqrt{-D} D^* +\frac{u_2}{2} (1+\sqrt{-D}D^*)\mod{N},$$ where $DD^* \equiv 1 \mod{N}$. Reordering terms we see that $$u \equiv \frac{1}{2}(1+\sqrt{-D})(u_2 D^*+2 u_1 D^*)+\frac{u_2}{2} (1-D^*)-u_1D^* \mod{N}.$$ If $N$ is odd then $2$ is invertible mod $N$, so we see that $u$ is  equivalent modulo $N$ to an element of $\Oo$. If $N$ is even then $DD^* \equiv 1 \mod{N}$ shows (together with $D$ odd) that $D^*$ is odd and we can make the same conclusion.  
 By a change of variable the sum is therefore equal to the right hand side of the statement of the Lemma.
\end{proof}

Following \cite{Haverkamp95} let us call the right hand side in the Lemma \ref{37} 
$\frac{1}{N}G_{-D}(a,N)$.
 Since  $ad-pbc=1$ ensures that $\textup{gcd} (a,pc)=1$ the Lemma shows that \eqref{claimgs} is equivalent to \begin{equation} \label{claimgs2} \frac{1}{pc} G_{-D}(a,pc)=\frac{1}{c} G_{-D}(a,c).\end{equation} For odd $c$ Lemma \ref{gauss} shows that both sides are equals to $\chi_K(c) $ (recall 
 that $p>2$ and $\chi_K(p)=1$ by assumption).

For general $c$ we argue as follows:
For $\textup{gcd} (c_1,c_2)=1$ it is easy to see that $$G_{-D}(a,c_1 c_2)=G_{-D}(ac_1, c_2) G_{-D}(ac_2, c_1).$$
By factoring $c=2^{e_2} p^{e_p} q$ with $(q,2p)=1$ we can rewrite \eqref{claimgs2}  as 
$$\frac{1}{pc}G_{-D}(2^{e_2} a, p^{e_p+1}q)G_{-D}(ap^{e_p+1}q,2^{e_2})=\frac{1}{c}G_{-D}(2^{e_2}a, p^{e_p}q) G_{-D}(ap^{e_p}q, 2^{e_2}).$$
Applying Lemma \ref{gauss} for the Gauss sums with odd second argument we see that we have reduced \eqref{claimgs2} to 
$$G_{-D}(ap^{e_p+1}q,2^{e_2})=G_{-D}(ap^{e_p}q,2^{e_2}).$$ This equality is true since $$G_{-D}(\bullet p,2^{e_2})=\sum_{\gamma \in \Oo/2^{e_2} \Oo} e[\frac{\bullet p}{2^{e_2}}|\gamma|^2]=\sum_{\gamma \in \Oo/2^{e_2} \Oo} e[\frac{\bullet}{2^{e_2}}|\pi \gamma|^2]=\sum_{\gamma \in \Oo/2^{e_2} \Oo} e[\frac{\bullet}{2^{e_2}}|\gamma|^2]$$ since $p>2$ and $|\pi|^2=p$.

For case (3.) we refer to the final part of Lemma \ref{trafo1}:
For $D \mid c$ the only terms in the expression for $M_{u,v}(\sigma)$ involving $b, c$ or $u$ are $\delta_{u, dv}$ and $ab |u|^2$. So when $b$ changes to $pb$ (and  $\pi u$ to $u$, and $\pi v$ to $v$)  the expressions for $M_{\pi u,\pi v}(\sigma)$ and $M_{u,v}(\bmat  p&0\\0&1\emat \sigma \bmat  p^{-1}&0\\0&1\emat )$ are equal.
\end{proof}

\section{Maass lift of $p$-old plusforms}\label{Maass lift of p-old plusforms}
Assume that $D=D_K$ is prime and $p$ is split in $K/\bfQ$. Set $S_{k-1}^+(Dp, \chi_K)$ to be the subspace of $M_{k-1}^+(Dp, \chi_K)$  (cf. Definition \ref{def3.4}) consisting of cusp forms. 
 The goal of this section is to prove the existence of a Maass lift for $p$-old forms in $S_{k-1}^+(Dp,\chi_K)$. This will allow us in section  \ref{A prime to p interpolation} to $p$-adically interpolate the Maass lift of ordinary newforms in $S_{k-1}(D, \chi_K)$. In \cite{Krieg91} Krieg defines the Maass lift for
$h \in M_{k-1}(D, \chi_K)$ by relating $h-h^c \in M_{k-1}^+(D, \chi_K)$ to a Jacobi form (thereby proving the surjectivity of the map in Proposition \ref{prop3.5} for $N=1$)
 and then invoking Theorem \ref{Jacobiisom} (again in the case $N=1$) to associate a Hermitian Maass form.

We briefly recall the key step in the construction of the Maass lift of \cite{Krieg91}:
Given $u \in \mD^{-1}$
and $g(\tau)=\sum_n a_n(g) e[\tau n] \in M_{k-1}^+(D,\chi_K)$
 define (as in \cite{Krieg91} 6(1))
 \begin{equation} \label{g_u2}  g_u(\tau)= \frac{-i \sqrt{D}}{a_D(-DN(u))}  \sum_{\substack{\ell \in \bfZ_{\geq 0}\\ -\ell \equiv DN(u) \mod{D}}} a_{\ell}(g) e[\ell \tau/D].\end{equation}
Krieg proves in the theorem in section 6 of \cite{Krieg91} that 
$$\varphi_g(\tau, z, w):=\sum_{u \in \mD^{-1}/\OK} g_u(\tau) \vartheta_u(\tau, z, w)$$ is a Jacobi form 
of weight $k$, index 1 and level ${\rm SL}_2(\bfZ)$, i.e., $\varphi_g \in J_{k,1}(1)$.

For a $p$-oldform   in $S_{k-1}^+(Dp,\chi_K)$ we will now modify Krieg's definition, but before we do so, we collect some of the properties of the forms lying in the plus-space in the following lemma.

\begin{lemma} \label{lem4.1}
\begin{enumerate}

\item  If $h \in S_{k-1}(D, \chi_K)$ is a normalized eigenform then $g:=h-h^c$  belongs to $S^+_{k-1}(D,\chi_K)$ and $g|_{k-1} \bmat  p&0\\0&1 \emat $ belongs to $S^+_{k-1}(Dp,\chi_K)$. 
\item The space $S^+_{k-1}(D,\chi_K)$ is generated by $h-h^c$ for normalised eigenforms $h \in S_{k-1}(D, \chi_K)$.
\item Any $p$-old form in $S_{k-1}^+(Dp, \chi_K)$ is of the form $\lambda g_1+\mu g_2|_{k-1}\bmat  p&0\\0&1 \emat $ for $g_1, g_2 \in  S_{k-1}^+(D, \chi_K)$ and $\lambda, \mu \in \bfC$.
\end{enumerate}
\end{lemma}

\begin{proof}
\begin{enumerate}

\item  This follows from the following formula \cite{Miyake89} (4.6.17): \begin{equation} \label{m4.6.17} a_n(h^c)=\ov{a_n(h)}=\chi_K(n) a_n(h) \text{ for } \textup{gcd}(n,D)=1.\end{equation}

\item As in \cite{Krieg91} p.671 we use for this that $S_{k-1}(D,\chi_K)$ has a basis of newforms $h_1, \ldots h_a, h_1^c, \ldots h_a^c, h_{a+1}, \ldots h_{a+b}$ with $h_i \neq h_i^c$ for $1 \leq i \leq a$ and $h_i=h_i^c$ for $a<i\leq a+b$. This implies the statement by \cite{Miyake89} Theorem 4.6.8(1), similar to the following argument for the $p$-old plusforms.
\item
Let $f \in S^+_{k-1}(Dp, \chi_K)^{p-{\rm old}}$. Then $f=f_1+f_2+f_3+f_4$, where $f_1\in \bigoplus_{i=1}^a \bfC(h_i-h_i^c)$, $f_2\in \bigoplus_{i=1}^{a+b} \bfC(h_i+h_i^c)$, $f_3\in \bigoplus_{i=1}^a\bfC (h_i-h_i^c)|_{k-1}\bsmat p \\ &1 \esmat$, $f_4 \in \bigoplus_{i=1}^{a+b}\bfC(h_i+h_i^c)|_{k-1}\bsmat p \\ &1\esmat$. 
Set $g=f_2+f_4$. We claim that the Fourier coefficients $a_n(g)=0$ for all $\textup{gcd}(n,D)=1$. By \cite{Miyake89} Theorem 4.6.8(1) this implies that $g=0$, which proves statement (3) of the lemma.

Consider first the case when $\chi_K(n)=1$. Then $a_n(g)=0$ since $g=f-(f_1+f_3) \in S_{k-1}^+(Dp, \chi_K)$ by assumption and (1).

If $\chi_K(n)=-1$ then $a_n(f_2)=0$ since $a_n(h+h^c)=0$ for any $h \in S_{k-1}(D, \chi_K)$
by \eqref{m4.6.17}. Write $n=n'p^r$ with $\textup{gcd}(n',p)=1$ and $r \geq 0$. If $r=0$ (i.e. $\textup{gcd}(n,p)=1$) then $a_n((h+h^c)|_{k-1}\bmat p \\ &1\emat)=0$
and if $r \geq 1$ then $$a_n((h+h^c)|_{k-1}\bmat p \\ &1\emat)=a_{n'p^{r-1}}(h+h^c),$$ which is zero again by \eqref{m4.6.17} since $\chi_K(n'p^{r-1})=\chi_K(n/p)=\chi_K(n)=-1$. Since $g=f_2+f_4$ this shows that $a_n(g)=0$ whenever $\chi_K(n)=-1$, concluding the proof.
\end{enumerate}
\end{proof}

For a $p$-oldform in $S_{k-1}^+(Dp,\chi_K)$ we modify Krieg's definition as follows: 
\begin{definition} \label{piu}
For $f=\lambda g^1 + \mu g^2|_{k-1} \bmat  p&0\\0&1\emat $ with $g^1, g^2 \in  S_{k-1}^+(D,\chi_K)$, 
$\pi \in \Oo$ with $N(\pi)=p$ and $g^i_u$ as in (\ref{g_u2}) we define (making use of the fact that  multiplication by $\pi$ induces a bijection on $\mD^{-1}/\Oo$)
$$f_{\pi u}:=\lambda g^1_{\pi u}+ \mu g^2_u|_{k-1} \bmat  p&0\\0&1\emat .$$
\end{definition}

\begin{prop} \label{varphi1}
 $$\varphi_f(\tau, z, w):=\sum_{u \in \mD^{-1}/\Oo} f_u(\tau) \vartheta_u(\tau, z, w)=\lambda \varphi_{g^1}(\tau, z, w)+ \mu \sum_{u \in \mD^{-1}/\Oo} g^2_u(p \tau) \vartheta_{\pi u}(\tau, z, w)$$ is a Jacobi form of  weight $k$, index 1 and  level $p$.
\end{prop}

\begin{proof}
We need to check that for all $\sigma \in \Gamma_0(p)$ we have  $$f_{\pi u}|_{k-1}\sigma=\sum_{v\in \mD^{-1}/\Oo} N_{\pi u,\pi v}(\sigma)f_{\pi v},$$ 
where the matrix $N(\sigma)=(N_{u,v}(\sigma))$ is defined in section \ref{Theta decomposition}. 
 It follows from Proposition \ref{general}  that 
$$N_{\pi u,\pi v}(\sigma)=N_{u,v}\left(\bmat  p&0\\0&1\emat \sigma \bmat  p^{-1}&0\\0&1\emat \right) $$ for all $\sigma \in \Gamma_0(p)$.
The proof of the theorem in section 6 of \cite{Krieg91}  (which can be used here because the $g^i$ are of level $D$) shows that $$g^i_u|_{k-1}\sigma=\sum_{v\in \mD^{-1}/\Oo} N_{u,v}(\sigma) g^i_v$$ for all $\sigma \in {\rm SL}_2(\bfZ)$.

Now we calculate that 
for $\sigma \in \Gamma_0(p)$
 \begin{equation*}\begin{split}f_{\pi u}|_{k-1}\sigma=& \lambda g^1_{\pi u}|_{k-1} \sigma + \mu g^2_u|_{k-1}\bmat p&0\\ 0&1 \emat \sigma\\
 =& \lambda g^1_{\pi u}|_{k-1}\sigma + \mu g^2_u|_{k-1}\left(\bmat p&0\\ 0&1 \emat \sigma \bmat p&0\\ 0&1 \emat^{-1}\right) \bmat p&0\\ 0&1 \emat\\
=&\lambda \sum_{v\in \mD^{-1}/\Oo} N_{\pi u,v}(\sigma) g^1_v + \mu \sum_{v \in \mD^{-1}/\Oo} N_{u,v}\left(\bmat  p&0\\0&1\emat \sigma \bmat  p^{-1}&0\\0&1\emat \right) g^2_v|_{k-1}\bmat  p&0\\0&1\emat \\
=&\sum_{v\in \mD^{-1}/\Oo} N_{\pi u,\pi v}(\sigma) \left(\lambda g^1_{\pi v} -  \mu g^2_v|_{k-1}\bmat  p&0\\0&1\emat \right).\end{split} \end{equation*}
\end{proof}

\begin{thm} \label{lift}
For $f \in S^+_{k-1}(Dp, \chi_K)^{p-{\rm old}}$  there exists   $F_f \in \mathcal{M}^*_k(p)$ with \be \label{Fc11} \alpha_{F_f}^*(\ell)=\sqrt{D}\frac{a_{\ell}(f)}{i a_D(\ell)}.\ee We will refer to $F_f$ as the \emph{Maass lift} of $f$. Furthermore, the assignment $f \mapsto F_f$ defines an injective  $\bfC$-linear map from $S^+_{k-1}(Dp, \chi_K)^{p-{\rm old}}$ to $\mM_{k}^*(p)$. 
\end{thm}

\begin{rem}
Due 
 partly to the absence of an old/newform theory for Hermitian forms it is non-trivial to extend the lift to $p$-oldforms. 
 In particular, for $f=\lambda g^1 + \mu g^2|_{k-1} \bmat  p&0\\0&1\emat $ our lift $F_f(Z)$ does not equal $\lambda F_{g^1}(Z) +\mu F_{g^2}(pZ)$ for the $F_{g^i}$ defined by \cite{Krieg91}.
\end{rem}

\begin{proof}
The linearity of the map follows immediately from (\ref{Fc11}) because $a_{\ell}(f)$ is linear in $f$. 
By Lemma \ref{lem4.1} (3) the form $f=\lambda g^1 + \mu g^2|_{k-1} \bmat  p&0\\0&1\emat $ with $g^1, g^2 \in  S_{k-1}^+(D,\chi_K)$. 
 Since $\varphi_f \in J_{k,1}(p)$ by Proposition \ref{varphi1}, it follows from Theorem \ref{Jacobiisom} that there exists a unique Maass form $F_f \in \mathcal{M}_k^*(p)$ corresponding to $\varphi_f$.

To prove (\ref{Fc11}) we claim that it suffices to show that $\alpha_{\varphi_f}^*(\ell)=\sqrt{D}\frac{a_{\ell}(f)}{i a_D(\ell)}$, i.e., to show that $\varphi_f$ maps to $f$ under the mapping of Proposition \ref{prop3.5} (with $N=p$),  using that $f$ lies in the plus-space.
Indeed we note that by combining (\ref{desc}) with the definition of $\alpha_{\varphi_f}^*$ in section \ref{Theta decomposition} we get for $T=\bmat \ell & t \\ \ov{t} & m\emat$ that \be \label{twoalphas} \alpha^*_{F_f}(D \det T/d^2) = c_{\varphi_f}\left(\frac{\ell m}{d^2}, \frac{t}{d}\right) = \alpha_{\varphi_f}^*\left(D\left(\frac{\ell m}{d^2}-\frac{|t|^2}{d^2}\right)\right).\ee 
We also note that it is enough to consider $\alpha_{\varphi_f}^*(\ell)$ for $\ell \equiv -DpN(u)$ (mod $D$) because of Remark \ref{2.1} and the fact that multiplication by the norm of an element $\alpha \in \Oo$ prime to $D$ induces a bijection on the set $N(\mD^{-1}/\Oo)$.

 We now calculate that analogous to (\ref{g_u2}) we have
\begin{equation} \label{f_u3} f_{\pi u}(\tau)= \frac{-i \sqrt{D}}{a_D(-DpN(u))}  \sum_{\ell \in \bfZ_+, -\ell \equiv DpN(u) \mod{D}} a_{\ell}(f) e[\ell \tau/D].\end{equation}
For this we express $f_u$ in terms of $g^i_u$ using  Definition \ref{piu} and utilize the ``Fourier expansion'' of $g^i_u$ given by (\ref{g_u2}) to get that for  $\ell \equiv -DpN(u) \mod{D}$ the coefficient of $e[\ell \tau/D]$ in the expansion of $f_{\pi u}$ is
$$-i \sqrt{D} \left(\frac{\lambda a_{\ell}(g^1)}{a_D(-DpN(u))} +  \frac{\lambda a_{\ell/p}(g^2)}{a_D(-DN(u))} \right).$$
By \cite{Krieg91} formula 4(5) we know that $$a_D(\ell) = \begin{cases} 1+\chi_K(-\ell) & D \nmid \ell\\ 0 & D \mid \ell. \end{cases}$$ Since $\textup{gcd} (p,D)=1$, we see that $D\mid DpN(u)$ if and only if $D \mid DN(u)$. Also, since $p$ is split we have $\chi_K(p)=1$, so for all $u \in \mD^{-1}$ we get $a_D(-DpN(u)) = a_D(-DN(u))$, which proves \eqref{f_u3}.

On the other hand, since $\varphi_f$ is a Jacobi form, we get a decomposition of $\varphi_f$ as in (\ref{decompJ}). Since such a decomposition is unique
 the $f_u$s in section \ref{Theta decomposition} coincide with the $f_u$s  considered in \eqref{f_u3} which enter  in the definition of $\varphi_f$ (cf. Proposition \ref{varphi1}). Thus we have by (\ref{fut}) that $$f_u = \sum_{\substack{ \ell \in \bfZ_+\\ \ell \equiv -DN(u) \hspace{1pt} \textup{mod} \hspace{1pt} D}} \alpha^*_{\varphi_f}(\ell) e[\ell \tau/D].$$ 
Comparing this with \eqref{f_u3} then implies that for $\ell \equiv -DpN(u) \mod{D}$ we have
$$\alpha_{\varphi_f}^*(\ell)=-i \sqrt{D} \left(\frac{a_{\ell}(f)}{a_D(-DpN(u))} \right).$$  The injectivity of the map from $f$ to $\varphi_f$ is clear since  we showed that it is the inverse to (the restriction to $S^+_{k-1}(Dp, \chi_K)^{p-{\rm old}}$ of) the map from Proposition \ref{prop3.5}. Combined with Theorem \ref{Jacobiisom}  this shows that the map $f \mapsto F_f$ must be injective.  This concludes the proof of the theorem. 
\end{proof}

\begin{rem} \label{rem4.6} Our construction is very different from that of Kawamura in \cite{Kawamura10preprint} 
who, in the setting of the 
Maass lift to Siegel modular forms, produces an analogous lift for Hecke eigenforms by $p$-stabilising the classical full level lift. 
As it is not clear which $p$-stabilisation procedure to follow for Hermitan forms we chose this more direct approach.  
Note that  our construction allows us to lift any oldform (as opposed to only eigenforms), i.e., is more in the spirit of Krieg \cite{Krieg91} and corresponds to what Ikeda calls a `linearized' lift   - cf. \cite{Ikeda08} sections 15 and 16 for full level version. 
However, if $f \in  S^+_{k-1}(Dp, \chi_K)^{p-{\rm old}}$ arises from an eigenform  $h \in S_{k-1}(D, \chi_K)$, then it follows from Proposition \ref{about Up} and Remark \ref{remdesc} 
that  our Maass lift $F_f$ is indeed  a Hecke eigenform (at least away from $D$)  with eigenvalues agreeing with those of the classical Maass lift studied by Krieg et al. 
of $h$ 
away from $p$ and $D$. The Maass lift  is semi-ordinary at $p$ 
provided that $h$ is ordinary 
 at $p$ (cf. section \ref{eigenforms}).
\end{rem}

\section{The Hecke invariance} \label{The Hecke invariance}
The goal of this section is to prove that the Maass space is invariant under the action of certain  Hecke operators. As before we assume that $\#\Cl_K=1$, $D_K$ is prime and $N$ is prime to $D_K$.

\subsection{The good primes} Let $p$ be a prime such that $p \nmid ND_K$. Consider $F \in \mM^*_k(N)$. A set of generators of the local Hecke algebra  $\mH_p$  at $p$ is given in sections 4.1.1 and 4.1.2 of \cite{Klosin15}. Since this case is almost identical to the level 1 case, we will not need the precise definitions here and instead refer the reader to \cite{Klosin15} for details.
\begin{prop} For any $T \in \mH_p$, one has $T F \in \mM_k^*(N)$. \end{prop}
\begin{proof}  In case when $p$ is inert (resp. split) in $K$, the proof is just a simple modification (consisting of making sure that the condition $\textup{gcd}(d,N)=1$ can be inserted in all the relevant spots) of the proof of Theorem 7 in \cite{Krieg91} (resp. of Theorem 5.10 in \cite{Klosin15}). 
\end{proof}

\subsection{The  primes dividing $N$} 

Suppose $p \mid N$. 
 In this section we will prove that $ \mM^*_k(N)$ is invariant under the Hecke operator $U_{p}:= \Gamma_0^{(2)}(N) \diag(1,1,p, p) \Gamma_0^{(2)}(N)$. Here the situation turns out to be simpler than in the case of good primes, but since to the best of our knowledge this case has never been specifically treated in the literature, we will include the proof. 
Then 
$$U_p = \Gamma_0^{(2)}(N) \bmat 1 \\ &1\\ &&p\\ &&&p \emat \Gamma_0^{(2)}(N)=\bigsqcup_{\substack{a,c \in \bfZ/p\bfZ\\ b \in \Oo/p\Oo}}\Gamma_0^{(2)}(N) \bmat 1&&a&b \\ &1&\ov{b}&c\\ &&p\\ &&&p \emat .$$

If we write $F(Z)=\sum_{T \geq 0} C_F(T) e[\tr TZ]\in \mM_k(N)$ 
 for the Fourier expansion of $F$, then we have \be \begin{split} (U_pF)(Z) =& p^{-2k}\sum_{T \geq 0} C_F(T)\sum_{\substack{a,c \in \bfZ/p\bfZ\\ b \in \Oo/p\Oo}} e\left[\tr\left( T(Z+\bmat a&b\\ \ov{b}&c \emat\right)p^{-1}\right]\\
=&  p^{-2k}\sum_{T \geq 0} C_F(T)e[\tr TZp^{-1}]\sum_{\substack{a,c \in \bfZ/p\bfZ\\ b \in \Oo/p\Oo}} e\left[\tr T\bmat a&b\\ \ov{b}&c \emat p^{-1}\right]\\
\end{split}\ee
Writing $T=\bmat n&\alpha \\ \ov{\alpha} & m\emat$ with $n,m \in \bfZ$ and $\alpha \in \mD^{-1}$ we see that the last sum equals $$\sum_{\substack{a,c \in \bfZ/p\bfZ\\ b \in \Oo/p\Oo}} e\left[\frac{na+\tr_{K/\bfQ}(\alpha \ov{b}) + mc}{p}\right]=\begin{cases} p^4 & \textup{if $p\mid \epsilon(T)$} \\ 0 & \textup{otherwise}.\end{cases}$$ 
Hence we conclude that \be \label{Upeq}  U_p\sum_{T \geq 0} C_F(T) e[\tr TZ] =p^{-2k+4} \sum_{T \geq 0} C_F(pT)e[\tr TZ].\ee

\begin{prop} \label{preserved} Suppose $F \in \mM_k^*(N)$. Then $U_p F \in \mM_k^*(N)$. \end{prop}
\begin{proof} Set $G:= U_pF$. Write $G(Z) = \sum_{T \geq 0} C_G(T)e[\tr TZ]$ for the Fourier expansion of $G$. We need to show that there exists a function $\alpha_G^*: \bfZ_{\geq 0} \to \bfC$ such that $C_G(T) = \sum_{\substack{d \in \bfZ_{+}\\ d \mid \epsilon(T)\\ \textup{gcd}(d,N)=1}} d^{k-1}\alpha_G^*(D_K \det T/d^2).$ Write $\alpha_F^*$ for the corresponding function for $F$.
Then we have \be \begin{split} (U_pF)(Z) =& p^{-2k+4}\sum_{T \geq 0} C_F(pT) e [\tr TZ]\\
=&p^{-2k+4} \sum_{T \geq 0} \sum_{\substack{d \in \bfZ_{+}\\ d \mid \epsilon(pT)\\ \textup{gcd}(d,N)=1}} d^{k-1} \alpha^*_F(D_K \det (pT)/d^2) e[\tr TZ]\\
= &p^{-2k+4}\sum_{T \geq 0}  \sum_{\substack{d \in \bfZ_{+}\\ d \mid \epsilon(T)\\ \textup{gcd}(d,N)=1}} d^{k-1} \alpha^*_F(p^2D_K \det T/d^2) e[\tr TZ],\end{split}\ee where the last equality comes from the fact that $\epsilon(pT) = p\epsilon(T)$ and since $p \mid N$, the condition $\textup{gcd}(d,N)=1$ forces the conditions $d \mid \epsilon(pT)$ and $d \mid \epsilon(T)$ to be equivalent. We can thus set $\alpha_G^*(x) := p^{-2k+4}\alpha_F^*(p^2x).$ \end{proof}

\subsection{Maass lifts of ordinary eigenforms} \label{eigenforms}
 If we use the ``arithmetic'' normalization of the Hecke action and scale the slashing operator $|_k \gamma$ by the additional factor of $\mu (\gamma)^{2k-4}$, 
then since $\mu(\diag(1,1,p,p)) = p$, the factor $p^{-2k+4}$ in (\ref{Upeq}) will disappear. Here $\mu$ denotes the similitude homomorphism defined on $\GU(2,2)$. Fix an embedding $\ov{\bfQ}_p \hookrightarrow \bfC$. We call a newform $h=\sum_{n=1}^{\infty} a_n(h) q^n \in S_{k-1}(D_K,\chi_K)$ $p$-ordinary if $\val_p(a_p(h))=0$. 
If this is the case then the Hecke polynomial $X^2 -a_p(h)X +p^{k-2}\chi_K(p)$ has two roots, one of which, say $\alpha$, is a $p$-adic unit, while the other, say $\beta$, is not.  Furthermore, the form  $f:= (h-h^c) - \beta (h-h^c)|_{k-1}\bmat p \\ &1\emat$ 
 lies in   $ S^+_{k-1}(D_Kp, \chi_K)^{p-\textup{old}}$. As usual 
  one can define the classical ``$U_p$'' operator on $S_{k-1}(D_Kp, \chi_K)$, which we denote here by $U(p)$ to distinguish it from the $U_p$ operator defined above, by setting $U(p)\sum_{n=1}^{\infty} a_n q^n:= \sum_{n=1}^{\infty} a_{np} q^n$. This operator preserves the plus-space if $p$ is split in $K/\bfQ$, while its square preserves it for all $p$. 

From now on suppose $p$ splits in $K/\bfQ$. Let $F_f\in \mM^*_{k}(p)$ be the Maass lift of $f$. We record the following result.
\begin{prop} \label{about Up} One has $U(p)f=\alpha f$ 
and $U_pF_f = \alpha^2 F_f$ where $U_p$ is normalized arithmetically. \end{prop}
\begin{proof} The first assertion is a standard fact for forms defined in the same way as $f$ but with $h$  in place of $h-h^c$. Thus it remains true for $f$ since $a_p(h^c)=\chi(p) a_p(h)=a_p(h)$ (cf. (\ref{m4.6.17})).  The second assertion follows from this, formula (\ref{Fc11}) and the fact that $\alpha^*_{U_pF_f}(\ell) = \alpha^*_{F_f}(p^2\ell)$ 
(cf. the last line of the  proof of Proposition \ref{preserved} combined with our ``arithmetic normalization'' of $U_p$) as well as the fact that $a_{D_K}(\ell) = a_{D_K}(p^2\ell)$.
 \end{proof}
\begin{rem} \label{remdesc} Proposition \ref{about Up} can be seen as asserting that the Hermitian Hecke operator $U_p$ `descends' to the elliptic Hecke operator $U(p)^2$. This is somewhat analogous to the case of Siegel modular forms, where the $U_p$ operator descends to $U(p)$ (i.e., no square) - this follows from combining Theorem 4.1 in \cite{Ibukiyama12} with the work of Kohnen \cite{Kohnen82} and Manickam et al. \cite{ManickamRamakrishnanVasudevan93} (we are grateful to Jim Brown for providing us with these references). Furthermore, Theorems 5.18 and 5.19 in \cite{Klosin15} imply that $F_f$ is also an eigenform for the Hecke algebras $\mH_{\ell}$ for primes $\ell \nmid D_Kp$ and provide analogous `descent formulas' for the standard generators of these algebras  (but we caution the reader that the notation used in \cite{Klosin15} conflicts with ours - in particular $U_p$ in \cite{Klosin15} is different from our $U_p$). \end{rem}

\section{A $p$-adic interpolation of the Maass lift} \label{A prime to p interpolation}
As before in this section we assume that $K$ has class number one, that its discriminant $D=D_K$ is prime and we fix a prime $p$ which splits in $K$. 
In this section we write $\Oo_K$ for the ring of integers of $K$ 
(since we reserve the notation $\Oo$ for a $p$-adic ring as defined below). 
From now on we fix embeddings $\ov{\bfQ} \hookrightarrow \ov{\bfQ}_p \cong \bfC$.
Let  
 $E \subset \ov{\bfQ}_p$ be a sufficiently large 
finite extension of $\bfQ_p$ and set $\Oo$ to be the valuation ring of $E$. 
Set $\Gamma =1+p\bfZ_p$ and $\Lambda_{\Oo}:= \Oo[[\Gamma]]$.  The $p$-adic cyclotomic character $\epsilon:G_{\bfQ} \to \bfZ_p^{\times}$ induces a canonical isomorphism $\Gal(\bfQ_{\infty}/\bfQ) \xrightarrow{\sim} \Gamma$, where $\bfQ_{\infty}$ is the unique $\bfZ_p$-extension of $\bfQ$ unramified away from $p$. Thus we can regard the tautological character $\Gamma \to \Lambda_{\Oo}^{\times}$ as a character of $G_{\bfQ}$ via the composite $G_{\bfQ} \to \Gamma \to \Lambda_{\Oo}^{\times}$. We will denote this composite by $\iota$. We will write $\omega$ for both the mod $p$ cyclotomic character and the associated Dirichlet character (Teichm{\"u}ller lift). 
 For a positive integer $r$, a $p^{r-1}$th root of unity $\zeta$, and an integer $m\geq 2$ we define a 
continuous $\Oo$-algebra 
homomorphism $$\nu_{m, \zeta}: \Lambda_{\Oo} \to \ov{\bfQ}_p,  \quad 1+p \mapsto \zeta(1+p)^{m-1}.$$

Let $k_0$ be a positive integer such that $\# \OK^{\times} \mid k_0$. Let $\tilde{f}_{k_0-1} \in S_{k_0-1}(Dp, \chi_K)$ be a normalized Hecke eigenform 
new at $D$  with $\tilde{f}_{k_0-1}^c \neq \tilde{f}_{k_0-1}$. We assume that $\tilde{f}_{k_0-1}$ is $p$-ordinary, i.e., that its $U(p)$-eigenvalue is a $p$-adic unit. 

By Theorem 1.4.1 of \cite{Wiles88} there exists a finite extension $\mK$ of $\textup{Frac}(\Lambda_{\Oo})$ and a primitive, normalized (i.e., $c_1=1$) $\bfI$-adic Hecke eigenform  $\tilde{\mF}=\sum_n c_n q^n \in \bfI[[q]]$ of tame level $D$ and character $\Psi$ 
such that for some extension $\nu'_{k_0-1}: \bfI\to \ov{\bfQ}_p$  of $\nu_{k_0-1, 1}$, one has
 \be \label{desc11}  \nu'_{k_0-1}(\tilde{\mF}):= \sum_n  \nu'_{k_0-1}(c_n) q^n = \tilde{f}_{k_0-1}.\ee Here $\bfI$ denotes the integral closure of $\Lambda_{\Oo}$ in $\mK$ and $\Psi:= \chi_K \omega^{k_0-2}$. We recall that such an $\tilde{\mF}$ has the  property that for almost all pairs $(k, \zeta)$ and all extensions $\nu'_{k, \zeta}: \bfI \to \ov{\bfQ}_p$ of $\nu_{k, \zeta}$ one has that $ \nu'_{k, \zeta}(\tilde{\mF})$ is a modular form of weight $k$, level $Dp^r$and character $\Psi\omega^{1-k}\chi_{\zeta} = \omega^{k_0-k-1}\chi_K\chi_{\zeta}$.  
The Dirichlet character $\chi_{\zeta}$ of conductor $p^r$ is defined by mapping the image of $1+p$ in $(\bfZ/p^r)^{\times}$ to $\zeta$.
 
  Note that we use a slightly different normalisation to that in \cite{Wiles88}, whose specialisation maps $\nu_{k, \zeta}$ send  $1+p$  to $\zeta (1+p)^{k-2}$ (cf. also the discussion in \cite{BCG} Remark 3.3). 
 We also note that $\tilde{\mF}$ is not necessarily a unique normalized primitive $\bfI$-adic  eigenform with the property (\ref{desc11}), but the $\Gal(\mK/\textup{Frac}(\Lambda_{\Oo}))$-conjugacy class of $\tilde{\mF}$ is unique. In other words here we pick a member of this Galois conjugacy which lifts $\tilde{f}_{k_0-1}$.

For $k \equiv k_0$ (mod $p-1$)  - so that $\omega^{k_0-k} \equiv 1$ - we have that 
$\nu'_{k-1}( \tilde{\mF}) \in S_{k-1}(Dp, \chi_K)$ is a normalized $p$-ordinary eigenform. 
As is well-known thanks to the work of Hida (cf. Theorem 2.1 in \cite{HIda86} and also Theorem 6.1 in \cite{BCG}), one can associate to $\tilde{\mF}$ 
 a continuous (for the meaning of continuity in this context cf. \cite{HIda86}, p. 557), absolutely irreducible Galois representation 
$$\rho_{\tilde{\mF}}: G_{\bfQ} \to \GL_2(\mK)$$ which is unramified outside $Np$ and for each prime $q \nmid Np$ satisfies
$$\det (1-\rho_{\tilde{\mF}}(\Frob_q)X) = 1-c_qX + (\Psi \iota)(\Frob_q)X^2 .$$  
In particular one has  \be \label{detform}\det \rho_{\tilde{\mF}} =\Psi \iota.\ee 
We will write $\mV$ for the space of this representation.

As it follows from \cite{Miyake89} Theorem 4.6.17 that there are no ordinary newforms of level $Dp$, 
 for almost every $k$ and every extension $\nu'_{k-1}$ of $\nu_{k-1,1}$ the form $\nu'_{k-1}(\tilde{\mF})$ is a $p$-stabilization of an ordinary  newform $h_{\nu'_{k-1}} = \sum_n a_n(h_{\nu'_{k-1}})q^n \in S_{k-1}(D, \chi_K)$. To ease notation, 
 as before, we will write $h'_{k-1}$ instead of $h_{\nu'_{k-1}}$ (the prime (here and below) indicating that the form  potentially depends on the choice of an extension $\nu'_{k-1}$ of $\nu_{k-1,1}$). We then define $f'_{k-1}$ by $$f'_{k-1}:=g'_{k-1} - \beta'_{k-1}g'_{k-1}|_{k-1}\bmat p&0\\0&1 \emat,$$ where $g'_{k-1}:= h'_{k-1}-(h'_{k-1})^c$ and $\beta'_{k-1} := p^{k-2}/\alpha'_{k-1}$ with $\alpha'_{k-1}$ the eigenvalue of 
 $U(p)$ corresponding to 
$\nu'_{k-1}(\tilde{\mF})$. In other words we go counter-clockwise on the following diagram:
$$\xymatrix{\{\nu'_{k-1}(\tilde{\mF})\in S_{k-1}(Dp, \chi_K)^{p-\textup{old}}\}_{\nu'_{k-1}} & \{f'_{k-1}\in S_{k-1}^+(Dp, \chi_K)^{p-\textup{old}}\}_{\nu'_{k-1}} \\ \{h'_{k-1}\in S_{k-1}(D, \chi_K)\}_{\nu'_{k-1}} \ar[u]^{p-\textup{stabilization}}\ar[r]& \{g'_{k-1}:=h'_{k-1}-(h'_{k-1})^c \in S^+_{k-1}(D, \chi_K)\}_{\nu'_{k-1}} \ar[u]_{p-\textup{stabilization}}}$$
We will prove in Theorem \ref{newfamily} below that $\{f'_{k-1}\}_{\nu'_{k-1}}$ 
 is indeed a $p$-adic analytic family. (Note that it follows from Lemma \ref{lem4.1} (1) that $f'_{k-1}$ belongs to $S_{k-1}^+(Dp,\chi_K)$.) Let us  remark here that since our construction of the Maass lift does not allow $p^2$ to divide the level, we have to limit ourselves here to $r=1$, i.e., $\zeta=1$. 
 This amounts in essence to constructing a family only over $\bfQ_p$-points of the weight space, as opposed to $\ov{\bfQ}_p$-points.

\begin{thm} \label{newfamily} There exists a finite set $\mA$ of $k\equiv k_0$ (mod $p-1$) and an $\bfI$-adic form $\mF =\sum_n b_n(\mF) q^n \in \bfI[[q]]$ which has the property that for every $n$ and every $k\equiv k_0$ (mod $p-1$) with $k\not\in \mA$, every extension $\nu'_{k-1}:\bfI \to \ov{\bfQ}_p$ of $\nu_{k-1,1}$  
maps the Fourier coefficient $b_n(\mF)$ to the $n$th Fourier coefficient of the form $f'_{k-1}$  with $f'_{k-1}$ defined as above. \end{thm}

\begin{proof} Fix $k$ and write $a_n$ for $a_n(h'_{k-1})$, the $n$th Fourier coefficient of $h'_{k-1}$. Let $\beta$ be the root of $X^2-a_p X + p^{k-2}$ which is not a $p$-adic unit.  
Then the $n$th Fourier coefficient $a_n(f'_{k-1})$ of  $f'_{k-1}$ equals $ (a_n-\ov{a_n}) - \beta (a_{n/p} - \ov{a_{n/p}})$. 
Writing $n=MD^sp^r$ with $pD\nmid M$ and using \eqref{m4.6.17} we get 
\be \begin{split} \label{Fc1}a_n(f'_{k-1})= & a_M a_{D^s}a_{p^r} - \chi_K(M)\chi_K(p^r)a_M \ov{a_{D^s}}a_{p^r} \\
-& \beta(a_M a_{D^s}a_{p^{r-1}}- \chi_K(M)\chi_K(p^{r-1})\chi_K(p)a_M \ov{a_{D^s}}a_{p^{r-1}})\\
= &a_M[a_{p^r}(a_{D^s}-\chi_K(M) \ov{a_{D^s}}) - \beta a_{p^{r-1}}(a_{D^s}-\chi_K(M) \ov{a_{D^s}})]\\
= &a_M(a_{D^s}-\chi_K(M) \ov{a_{D^s}})(a_{p^r}-\beta a_{p^{r-1}})
\end{split}\ee
Note that $a_M(a_{p^r}-\beta a_{p^{r-1}})$ is the $Mp^r$-th Fourier coefficient of the $p$-stabilization $\nu'_{k-1}(\tilde{\mF})$ of $h'_{k-1}$, i.e., $a_M(a_{p^r}-\beta a_{p^{r-1}})=\nu'_{k-1}(c_{Mp^r})$.

To deal with $\ov{a_{D}^s}$ we note that since the $D$-eigenvalue $a_{D}(h'_{k-1})$ of $h'_{k-1}$ satisfies \be \label{4.6.17} a_{D}(h'_{k-1})\ov{a_{D}(h'_{k-1})}=D^{k-2}\ee

by Theorem 4.6.17(1) in \cite{Miyake89}, we see that condition (2) in Lemma 2.6.2 in \cite{EmertonPollackWeston06} is not satified for the height one prime corresponding to $h'_{k_0-1}$. Hence by that lemma  we get that the quotient $\mV/I_{D}\mV$ is an (unramified) rank one $\bfI$-module   on which $\Frob_{D}$ acts via the eigenvalue of $\tilde{\mF}$ at $D$. We denote this character by $\mu$ and we have $\mu(\Frob_{D})=c_{D}$. 
Set 
$$b_{MD^sp^r}(\mF):= c_{Mp^r}(\mu(\Frob_D)^s - \chi_K(M)(\iota \omega^{k_0-2} \mu^{-1})(\Frob_D)^s)\in \bfI.$$

The theorem now follows  from (\ref{4.6.17})  as 
 $(\iota \omega^{k_0-2})(\Frob_D)$ specializes to $D^{k-2}$ 
 for $k\equiv k_0$ (mod $p-1$). 
\end{proof}

From now on assume that $p\equiv 1$ (mod $\#\OK^{\times}$) and assume that $k_0 \in \{0,1, \dots, p-2\}$ such that $\#\OK^{\times} \mid k_0$. Then for every $k \equiv k_0$ (mod $p-1$) we will have $\#\OK^{\times} \mid k$.

\begin{lemma} Let $d$ be an integer prime to $p$. There exists a power series $A_d(T) \in \bfZ_p[[T]]$ such that 
$A_d((1+p)^{k}-1) = \omega(d)^{-k} d^{k-1}$. \end{lemma} 

\begin{proof} This is stated on page 197 in \cite{Hida93}. \end{proof}

\begin{thm} \label{LambdaMaass} Let $k_0$ be as above. Let $\mF \in \bfI[[q]]$ be the $\bfI$-adic form in Theorem  \ref{newfamily}. There exists a finite set $\mA$ of $k \equiv k_0$ (mod $p-1$) with the property  that  for every $M\in \mS$, there exists $\mC_{F_{\mF}}(M) \in \bfI$ such that for all $k \equiv k_0$ (mod $p-1$) with $k \not\in \mA$ and all extensions $\nu'_{k-1}: \bfI \to \ov{\bfQ}_p$ of $\nu_{k-1,1}$ one has that $\nu'_{k- 1}(\mC_{F_{\mF}}(M))$ is the $M$th Fourier coefficient of the Maass lift of $\nu'_{k-1}(\mF)$. 
Thus the formal power series $$F_{\mF}:= \sum_{M \in \mS} \mC_{F_{\mF}}(M) q^M\in \bfI[[q]]$$ can be regarded as a $\bfI$-adic Maass lift of the family $\mF$. \end{thm}

\begin{proof} By Theorem \ref{newfamily} we know that $b_n(\mF)\in \bfI$ specializes under extension $\nu'_{k-1}$ to the $n$th Fourier coefficient of $f'_{k-1}$. Combining this with Theorem \ref{lift} we obtain $B_n(\mF):= -i\frac{\sqrt{D}}{a_D(n)} b_n(\mF) \in \bfI$ interpolating the values of the function $\alpha^*_{F_f}$. 
Set $A_{k_0,d}(T):= \omega(d)^{k_0}A_d(T) \in \bfZ_p[[T]]$ and let $\tilde{A}_{k_0,d}$ be the element of $\bfZ_p[[\Gamma]]$ corresponding to it under the isomorphism $T^k \mapsto (1+p)^k-1$. Set $$\mC_{F_{\mF}}(M):= \sum_{\substack{d \in \bfZ_+\\ d\mid \epsilon(M)\\ \textup{gcd}(d, p)=1}} A_{k_0,d}B_{D_K \det M/d^2}(\mF).$$
\end{proof}

\bibliographystyle{amsalpha}

\bibliography{standard2}

\end{document}

%% file: makros.tex
\def\a{\alpha}

\newcommand{\mA}{\mathcal{A}}

\newcommand{\mC}{\mathcal{C}}
\newcommand{\mD}{\mathcal{D}}

\newcommand{\mF}{\mathcal{F}}

\newcommand{\mH}{\mathcal{H}}

\newcommand{\mK}{\mathcal{K}}

\newcommand{\mM}{\mathcal{M}}

\newcommand{\mS}{\mathcal{S}}

\newcommand{\mV}{\mathcal{V}}

\newcommand{\fa}{\mathfrak{a}}

\newcommand{\fD}{\mathfrak{D}}

\newcommand{\fq}{\mathfrak{q}}

\newcommand{\bfC}{\mathbf{C}}

\newcommand{\bfH}{\mathbf{H}}
\newcommand{\bfI}{\mathbf{I}}

\newcommand{\bfQ}{\mathbf{Q}}
\newcommand{\bfR}{\mathbf{R}}

\newcommand{\bfZ}{\mathbf{Z}}

\newcommand{\Oo}{\mathcal{O}}

\newcommand{\OK}{\mathcal{O}_K}

\newcommand{\ov}{\overline}

\newcommand{\be}{\begin{equation}}
\newcommand{\ee}{\end{equation}}
\newcommand{\bes}{\begin{equation*}}
\newcommand{\ees}{\end{equation*}}

\newcommand{\bs}{\begin{split}}
\newcommand{\es}{\end{split}}
\newcommand{\bss}{\begin{split*}}
\newcommand{\ess}{\end{split*}}

\newcommand{\bmat}{\left[ \begin{matrix}}
\newcommand{\emat}{\end{matrix} \right]}
\newcommand{\bsmat}{\left[ \begin{smallmatrix}}
\newcommand{\esmat}{\end{smallmatrix} \right]}

\newcommand{\bml}{\begin{multline}}
\newcommand{\eml}{\end{multline}}
\newcommand{\bmls}{\begin{multline*}}
\newcommand{\emls}{\end{multline*}}

\DeclareMathOperator{\Cl}{Cl}

\DeclareMathOperator{\diag}{diag}

\DeclareMathOperator{\Frob}{Frob}
\DeclareMathOperator{\Gal}{Gal}
\DeclareMathOperator{\GL}{GL}
\DeclareMathOperator{\GSp}{GSp}
\DeclareMathOperator{\GU}{GU}

\DeclareMathOperator{\Res}{Res}

\DeclareMathOperator{\SL}{SL}

\DeclareMathOperator{\U}{U}
\DeclareMathOperator{\val}{val}

\def\c{\chi}

\newcommand{\hs}{\hspace{2pt}}

\newcommand{\tr}{\textup{tr}\hspace{2pt}}

%% file: settings.tex
\usepackage[all]{xy}
\usepackage{amsthm}
\usepackage{amssymb, amsfonts}
\SelectTips{cm}{10}\UseTips
\theoremstyle{plain}
\newtheorem{thm}{Theorem}
\newtheorem{prop}[thm]{Proposition}

\newtheorem{lemma}[thm]{Lemma}

\theoremstyle{definition}
\newtheorem{definition}[thm]{Definition}

\newtheorem{rem}[thm]{Remark}

\numberwithin{thm}{section}
\numberwithin{equation}{section}